\documentclass[11pt,oneside,reqno]{amsart}

\usepackage{amssymb}
\usepackage{algpseudocode}
\usepackage{algorithm}
\usepackage{subcaption}
\usepackage{tikz}
\usepackage[singlelinecheck=false]{caption}
\usepackage{verbatim}
\usepackage{url}
\usepackage{graphicx}
\usepackage[noadjust]{cite}
\RequirePackage{dsfont} 
 \scrollmode
\allowdisplaybreaks
\usepackage{graphicx}
\usepackage{epstopdf}

\definecolor{red-new}{rgb}{1,0,0}
\definecolor{dark-blue}{rgb}{0,0,1}
\definecolor{light-blue}{rgb}{0,0.6,0.6}
\definecolor{gray-new}{rgb}{0.3,0.3,0.3}

\usepackage{amsmath}
\usepackage{tikz}

\makeatletter
\newcommand{\xleftrightarrow}[2][]{\ext@arrow 3359\leftrightarrowfill@{#1}{#2}}
\makeatother

\newcommand{\xdasharrow}[2][->]{
\tikz[baseline=-\the\dimexpr\fontdimen22\textfont2\relax]{
\node[anchor=south,font=\scriptsize, inner ysep=1.5pt,outer xsep=2.2pt](x){#2};
\draw[shorten <=3.4pt,shorten >=3.4pt,dashed,#1](x.south west)--(x.south east);
}
}

\newcommand{\DEBUG}{}

\ifdefined\DEBUG%

  \def\rem#1{{\marginpar{\raggedright\scriptsize #1}}}
  \newcommand{\pmr}[1]{\rem{\color{blue}{$\bullet$ #1}}}
  \newcommand{\ppr}[1]{\rem{\color{red}{$\bullet$ #1}}}
 \else%

  \newcommand{\ppr}[1]{}
  \newcommand{\pmr}[1]{}
 \fi

\theoremstyle{plain}
\newtheorem{theorem}{Theorem}
\newtheorem{lemma}{Lemma}
\newtheorem{fact}{Fact}

\theoremstyle{definition}
\newtheorem{remark}{Remark}

\usepackage{enumitem}
\begin{document}

\title
[Randomized implicit RK2 schemes]
{Convergence and stability of randomized implicit two-stage Runge-Kutta schemes}

\author[T. Bochacik]{Tomasz Bochacik}
\address{AGH University of Krakow,
Faculty of Applied Mathematics,\newline
Al. A.~Mickiewicza 30, 30-059 Krak\'ow, Poland}
\email{bochacik@agh.edu.pl, corresponding author}

\author[P. Przyby{\l}owicz]{Pawe{\l} Przyby{\l}owicz}
\address{AGH University of Krakow,
Faculty of Applied Mathematics,\newline
Al. A.~Mickiewicza 30, 30-059 Krak\'ow, Poland}
\email{pprzybyl@agh.edu.pl}

\begin{abstract}
We randomize the implicit two-stage Runge-Kutta scheme in order to improve the rate of convergence (with respect to a deterministic scheme) and stability of the approximate solution (with respect to the solution generated by the explicit scheme). For stability analysis, we use Dahlquist's concept of A-stability, adopted to randomized schemes by considering three notions of stability: asymptotic, mean-square, and in probability. The randomized implicit RK2 scheme proves to be A-stable asymptotically and in probability but not in the mean-square sense. 
\end{abstract}
\maketitle
\tableofcontents

In this paper, we investigate two randomized implicit schemes from the two-stage Runge-Kutta family. The study of randomized algorithms approximating the solutions of initial value
problems for ODEs dates back to the early 1990s, see \cite{stengle1, stengle2}. There is a rich literature on error bounds for randomized explicit methods, including Euler, Taylor, and Runge-Kutta schemes \cite{randRK, daun1, HeinMilla, JenNeuen, Kac1, KruseWu_1} but just a few papers dealing with implicitness are available, see \cite{randEuler, backward_euler} where the randomized Euler scheme has been investigated. This paper aims to provide insights on higher-order randomized implicit methods. In section \ref{sec:prel}, we propose two ways of randomizing the implicit two-stage Runge-Kutta scheme. In the next sections, we deal with the convergence and A-stability of these schemes. 

From existing literature, it is known that randomization helps to achieve a better rate of convergence of a numerical scheme. Under standard assumptions (H\"older continuity in time and Lipschitz continuity in space), the rate of convergence typically improves by $1/2$ (additively) in comparison to the deterministic scheme. This general observation, verified by various authors for many other randomized schemes, holds also for the two schemes considered in this paper. The details are included in section \ref{sec:error}.

The main focus of this paper is stability. In section \ref{sec:stability}, we investigate the behavior of the randomized implicit two-stage Runge-Kutta scheme when applied to the famous Dahlquist's test equation, see \cite{dahlquist}. This simple test problem has been found particularly useful in the stability analysis for methods approximating solutions of ODEs. For details, we refer the reader to \cite{butcher, ambrosio, hairer}, where the concept of A-stability is comprehensively covered for deterministic schemes, with a particular focus on Runge-Kutta methods.

In the 1990s, the first papers on the stability of random numerical schemes were published, see for example \cite{higham1, higham2, mitsui}, however only for stochastic differential equations. In these papers, regions of mean-square and asymptotic stability were considered, by analogy to the concept of absolute stability region but with specifying in which probabilistic sense the approximated solution tends to $0$.

Surprisingly, for a long time, the stability of randomized methods for ODEs had not attracted attention. The first steps were made in \cite{taylor, randRK, randEuler}, where the notions of mean-square and asymptotic stability were adapted from the SDEs case to the randomized schemes for ODEs. Additionally, stability in probability has been considered. So far, these three types of probabilistic stability have been characterized for the family of randomized Taylor schemes (including the randomized explicit RK2 scheme) and for randomized Euler schemes. The current paper, according to our best knowledge, is the first effort toward stability analysis for higher-order implicit randomized schemes. 

In section \ref{sec:stability}, we show that:
\begin{itemize}
    \item stability regions for both schemes considered in this paper coincide;
    \item the schemes are asymptotically A-stable and A-stable in probability;
    \item however, they are not mean-square A-stable.
\end{itemize}
The mean-square stability region is further investigated and its basic properties are formally proven. We note that stability analysis for the randomized explicit two-stage Runge-Kutta scheme has been performed in \cite{randRK}. We show that the implicit algorithms considered in this paper outperform that scheme in terms of stability. This was expected based on well-known results for deterministic schemes. Finally, we compare the performance of the randomized implicit and semi-implicit RK2 schemes to their deterministic counterparts. In section \ref{sec:experiments} we conduct a numerical experiment, where the schemes are applied to a sample stiff problem. We observe that trajectories generated by the randomized schemes are contaminated by some noise (contrary to the deterministic approximation) but the error does not explode. We think that this behavior can be linked to lacking the mean-square A-stability and having the asymptotic A-stability property.

The main results of the paper are summarized in section \ref{sec:conclusions}. We discuss also related topics and hypotheses that may be considered for further research.

\section{Preliminaries} \label{sec:prel}

\subsection{Initial problem}

We deal with initial value problems of the following form: 
\begin{equation*}
		\left\{ \begin{array}{ll}
			z'(t)= f(t,z(t)), \ t\in [a,b], \\[2pt]
			z(a) = \eta, 
		\end{array}\right.
\end{equation*}
where $-\infty < a < b < \infty$, $\eta\in\mathbb{R}^d$, $f\colon [a,b]\times\mathbb{R}^d\to\mathbb{R}^d$, $d\in\mathbb{Z}_+$. 

\subsection{Schemes}

Let $n\in\mathbb{Z}_+$ and $h=\frac{b-a}{n}$, $t_j = a+jh$ for $j\in\{0,1,\ldots,n\}$. Moreover, let $\tau_1,\tau_2,\ldots$ be independent random variables uniformly distributed on the interval $[0,1]$, defined on a common probability space $(\Omega,\Sigma,\mathbb{P})$, where the $\sigma$-algebra $\Sigma$ is complete. For $j\in\{1,\ldots,n\}$, let $\theta_j = t_{j-1} + h\tau_j$. We consider also a filtration $(\mathcal{F}_j)_{j=0}^\infty$ generated by random variables $\tau_1,\tau_2,\ldots$. Specifically, $\mathcal{F}_0 = \sigma(N)$, where $N=\{ A\in\Sigma\colon \mathbb{P}(A)=0\}$, and $\mathcal{F}_j = \sigma\bigl( \sigma(\tau_1,\ldots,\tau_j), N \bigr)$. 

The classical deterministic two-stage Runge-Kutta scheme (also referred to as the midpoint scheme) is given by 
\begin{equation} \label{eq:detRK2}
    \left\{ \begin{array}{l}
			 V^0 = \eta,\\[2pt]
			 V_{1/2}^j =  V^{j-1} +\frac12 h   f \bigl(t_{j-1},   V^{j-1}\bigr), \ j\in\{1,\ldots,n\}, \\[2pt]
			 V^j =  V^{j-1} + h  f(t_{j-1}+\frac12h, V_{1/2}^j), \ j\in\{1,\ldots,n\}.
		\end{array}\right. \tag{det RK2}
\end{equation}
In \cite{HeinMilla, KruseWu_1, randRK}, its randomized version (the randomized explicit two-stage Runge-Kutta scheme or the random point scheme) has been considered:
\begin{equation} \label{eq:explRK2}
    \left\{ \begin{array}{l}
			 V^0 = \eta,\\[2pt]
			 V_{\tau}^j =  V^{j-1} +\tau_j h   f \bigl(t_{j-1},   V^{j-1}\bigr), \ j\in\{1,\ldots,n\}, \\[2pt]
			 V^j =  V^{j-1} + h  f(\theta_j, V_{\tau}^j), \ j\in\{1,\ldots,n\}.
		\end{array}\right. \tag{rand expl RK2}
\end{equation}
Both these schemes are explicit. The rate of convergence of \eqref{eq:explRK2} was shown to be $\varrho+\frac12$ (where $\varrho$ is the H\"older exponent in time) in \cite{HeinMilla, KruseWu_1}. Stability regions of \eqref{eq:detRK2} and \eqref{eq:explRK2} have been characterized and compared in \cite{randRK}.

In the current paper, we deal with randomized two-stage Runge-Kutta schemes that involve implicitness. We will investigate the following two schemes:
\begin{enumerate}
    \setlength{\itemsep}{2pt}
    \item[I.] \textit{Semi-implicit randomized RK2 scheme}, where the intermediate step is implicit:
        \begin{equation} \label{eq:S1}
        \left\{ \begin{array}{l}
            V^0 = \eta,\\[2pt]
            V_{\tau}^j =  V^{j-1} +\tau_j h   f (\theta_{j},   V_\tau^{j}), \ j\in\{1,\ldots,n\}, \\[2pt]
            V^j =  V^{j-1} + h  f(\theta_j, V_{\tau}^j), \ j\in\{1,\ldots,n\}.
        \end{array}\right. \tag{S$1$}
        \end{equation}
        The semi-implicit randomized RK2 scheme is conceptually similar to the explicit randomized RK2 scheme. Both schemes use the Euler-style approximation with step size $\tau_jh$ as an intermediate step: explicit Euler is employed for randomized explicit RK2 and implicit Euler for randomized semi-implicit RK2.
    \item[II.] \textit{Implicit randomized RK2 scheme}, where in $j$-th iteration, the slope is calculated at a~point chosen randomly from the segment joining $(t_{j-1},V^{j-1})$ and $(t_j, V^j)$:
        \begin{equation} \label{eq:S2}
        \left\{ \begin{array}{l}
            V^0 = \eta,\\[2pt]
            V^j =  V^{j-1} + h  f(\theta_j, (1-\tau_j)V^{j-1}+\tau_jV^j), \ j\in\{1,\ldots,n\}.
        \end{array}\right. \tag{S$2$}
        \end{equation}
\end{enumerate}
Sometimes both of the above schemes will be jointly referred to as randomized implicit RK2 schemes. Their deterministic counterparts can be recovered by setting $\tau_j=\frac12$ for all $j\in\mathbb{Z}_+$ in \eqref{eq:S1}:
    \begin{equation} \label{eq:detS1}
        \left\{ \begin{array}{l}
            V^0 = \eta,\\[2pt]
            V_{1/2}^j =  V^{j-1} +\frac12 h   f (t_{j-1}+\frac{h}{2},   V_{1/2}^{j}), \ j\in\{1,\ldots,n\}, \\[2pt]
            V^j =  V^{j-1} + h  f(t_{j-1}+\frac12h, V_{1/2}^j), \ j\in\{1,\ldots,n\}
        \end{array}\right. \tag{det S$1$}
        \end{equation}
and in \eqref{eq:S2}:
        \begin{equation} \label{eq:detS2}
        \left\{ \begin{array}{l}
            V^0 = \eta,\\[2pt]
            V^j =  V^{j-1} + h  f(t_{j-1}+\frac{h}{2}, \frac12 V^{j-1}+\frac12 V^j), \ j\in\{1,\ldots,n\}.
        \end{array}\right. \tag{det S$2$}
        \end{equation}
        
Whenever there is a risk of confusion, we will add a subindex to $V$ indicating which scheme is considered.

\section{Error bounds} \label{sec:error}

\subsection{Class of initial value problems}

We consider a class $F^\varrho =F^\varrho (a,b,d,K,L)$ of pairs $\left(\eta,f\right)$ satisfying the following conditions:
\begin{description}
\setlength\itemsep{4pt}
\item[A0\label{itm:A0}] $\left\|\eta\right\|\leq K$, 
\item[A1\label{itm:A1}] $f\in \mathcal{C}\left([a,b]\times\mathbb{R}^d\right)$,
\item[A2\label{itm:A2}] $\| f(t,x)\| \leq K\left(1+\left\|x\right\|\right)$ for all $(t,x)\in [a,b]\times\mathbb{R}^d$,
\item[A3\label{itm:A3}] $\| f(t,x) - f(s,x) \|\leq L|t-s|^\varrho$ for all $t,s\in [a,b], x\in \mathbb{R}^d$,
\item[A4\label{itm:A4}] $\| f(t,x) - f(t,y) \|\leq L\|x-y\|$ for all $t\in [a,b], x,y \in \mathbb{R}^d$,
\end{description}
where $\varrho\in (0,1]$ and $K,L \in (0,\infty)$. 

\subsection{Existence of the solution}

Each iteration of schemes \eqref{eq:S1} and \eqref{eq:S2} requires solving an equation for $V_\tau^j$ or $V^j$, respectively. We will show that for sufficiently small $h$, the equation has a solution and thus the scheme is well-defined. 

\begin{lemma} \label{lm:s1}
    Let $\left(\eta,f\right)\in F^\varrho$ and let us assume that $Lh<1$. Then for every $j\in\{1,\ldots,n\}$  there exists a solution $V_\tau^j$ to the equation
    $$V_{\tau}^j =  V_{\text{S}1}^{j-1} +\tau_j h   f (\theta_{j},   V_\tau^{j}),$$
    cf. \eqref{eq:S1}. This solution is unique (up to a null event) and $\mathcal{F}_j$-measurable.
\end{lemma}

\noindent We skip the proof of Lemma \ref{lm:s1} as it is analogous to the proof of the following Lemma \ref{lm:s2}.

\begin{lemma} \label{lm:s2}
    Let $\left(\eta,f\right)\in F^\varrho$ and let us assume that $Lh<1$. Then for every $j\in\{1,\ldots,n\}$  there exists a solution $V^j_{\text{S}2}$ to the equation
    \begin{equation} \label{eq:lm-s2}
        V_{\text{S}2}^j =  V_{\text{S}2}^{j-1} + h  f(\theta_j, (1-\tau_j)V_{\text{S}2}^{j-1}+\tau_jV^j),
    \end{equation}
    cf. \eqref{eq:S2}. This solution is unique (up to a null event) and $\mathcal{F}_j$-measurable.
\end{lemma}

\begin{proof}
    For brevity, $V_{\text{S}2}$ will be denoted by $V$ (with no subindex) in this proof.  We proceed by induction. $V^0=\eta$ is deterministic and thus $\mathcal{F}_0$-measurable. Let us consider $j\in\{1,\ldots,n\}$ and assume that $V^0,\ldots,V^{j-1}$ have been already constructed and satisfy the statement of the lemma. Let 
    $$M_j = \{\omega\in\Omega\colon \tau_j\in[0,1]\}.$$
    It is clear that $M_j\in\mathcal{F}_j$ and $\mathbb{P}(M_j)=1$. Let us define a mapping
    $$p_j \colon \Omega\times\mathbb{R}^d \ni (\omega,x) \mapsto V^{j-1}(\omega) + h  f\bigl(\theta_j(\omega), (1-\tau_j(\omega))V^{j-1}(\omega)+\tau_j(\omega)x\bigr) \in\mathbb{R}^d. $$
    By Assumption \ref{itm:A4}, for every $\omega\in M_j$ and every $x,y\in\mathbb{R}^d$,
    $$\|p_j(\omega,x)-p_j(\omega,y)\|\leq Lh \| x- y\|.$$
    Since $Lh<1$, the function $x\mapsto p_k(\omega,x)$ is a contraction for every $\omega\in M_j$. By Banach fixed-point theorem, for every $\omega\in M_j$ there exists a unique $V^j(\omega)\in\mathbb{R}^d$ such that $V^j(\omega)=p_j\bigl(\omega,V^j(\omega)\bigr)$. For $\omega \in\Omega \setminus M_j$ we can assign any value to $V^j(\omega)$ and finally we get a mapping $V^j\colon \Omega\to\mathbb{R}^d$, which is a solution (with probability $1$) to equation \eqref{eq:lm-s2} and this is the only solution (up to a null event). 
    
    Now we will show that $V^j$ is $\mathcal{F}_j$-measurable. Let us define
    \begin{align*}
        v_0(\omega) & = V^{j-1}(\omega), \\
        v_{k}(\omega) & = p_j (\omega,v_{k-1}(\omega)) = V^{j-1}(\omega) + h  f\bigl(\theta_j(\omega), (1-\tau_j(\omega))V^{j-1}(\omega)+\tau_j(\omega)v_{k-1}(\omega)\bigr)
    \end{align*}
    for $\omega\in\Omega$ and $k\in\mathbb{Z}_+$. By induction with respect to $k$, we may show that $v_k$ is $\mathcal{F}_j$-measurable for any $k\in\mathbb{Z}_+$ (it follows mainly from the fact that the function $f$ is Borel measurable, cf. Assumption \ref{itm:A1}). According to the Banach fixed-point theorem, 
    $$V^j(\omega) = \lim_{k\to\infty} v_k(\omega)$$ 
    for every $\omega\in M_j$. Since the probability space $(\Omega, \mathcal{F}_j,\mathbb{P})$ is complete, the limit (in the almost sure sense) $V^j$ of the sequence of $\mathcal{F}_j$-measurable random variables $v_k$ ($k\to\infty$) is also $\mathcal{F}_j$-measurable. This completes the proof.
\end{proof}

\subsection{Results on the $L^p(\Omega)$-error}

In the following two theorems, we establish error bounds for schemes \eqref{eq:S1} and \eqref{eq:S2}. Our results imply that randomization contributes an extra $1/2$ to the rate of convergence of the implicit RK2 scheme. This is consistent with previous results concerning the impact of randomization on the convergence rate of the explicit RK2 scheme \cite{randRK, KruseWu_1, HeinMilla}. Proof techniques are also similar.

\begin{theorem} \label{theorem-error-1}
Let $p\in[2,\infty)$. There exists a constant $C = C(a,b,d,K,L,\varrho, p)>0$ such that for every $n\geq \lfloor b-a\rfloor+1$ satisfying the condition $Lh<1$ and for every $(\eta,f)\in F^{\varrho}$, it holds that
\begin{equation}
    \label{eq:theorem-error-1}
    \Bigl\| \max_{0\leq j \leq n} \bigl\| z(\eta,f)(t_j)- V^j_{\text{S}1}(\eta,f)\bigr\| \Bigr\|_p \leq C h^{\varrho+\frac12}.
\end{equation}
\end{theorem}
\begin{proof}
    Since $z(t_0) = V^0 = \eta$ and $z'(\theta_j)=f(\theta_j,z(\theta_j))$, for every $k\in\{1,\ldots,n\}$ we get
\begin{align}
    z(t_k)-V^k & = \sum_{j=1}^k \bigl(z(t_j)-z(t_{j-1})\bigr) - \sum_{j=1}^k (V^j-V^{j-1}) \notag 
    \\ & = \sum_{j=1}^k \Bigl( \int\limits_{t_{j-1}}^{t_j} z'(s)\,\mathrm{d}s - hz'(\theta_j)\Bigr)  + h\sum_{j=1}^k \bigl[ f\bigl(\theta_j,z(\theta_j)\bigr)  - f\bigl(\theta_j,V_\tau^j\bigr) \bigr] \notag 
    \\ & = S_1^k + S_2^k. \label{eq:s1s2s3s4}
\end{align}
By Theorem 3.1 in \cite{KruseWu_1}, there is a constant $C=C(a,b,d,K,L,p)>0$ such that
\begin{equation} \label{eq:semi:bound_S1}
\Bigl\|\max_{1\leq j\leq n} \|S_1^j\|\Bigr\|_p \leq Ch^{\varrho+\frac12}.
\end{equation}

Let us define $u_0=0$ and $$u_k = \max_{0\leq j\leq k} \|z(t_j)-V^j\|=  \max_{1\leq j\leq k} \|z(t_j)-V^j\|$$ for $k\in\{1,\ldots,n\}$. By Assumption \ref{itm:A4} we get
\begin{equation}
    \label{eq:semi:s2}
    \max_{1\leq j\leq k}\|S_2^j\| \leq hL \sum_{j=1}^k \bigl\| z(\theta_j)-V_\tau^j \bigr\|.
\end{equation}
Note that
\begin{align*}
    \|z& (\theta_{j})-V_\tau^j\| = \bigl\| z(\theta_j)-V^{j-1} - \tau_jh f(\theta_j,V^j_\tau) \bigr\|
    \\ & \leq \bigl\| z(\theta_j)-z(t_{j-1})-\tau_jh f(\theta_j,z(\theta_j))\bigr\| + \|z(t_{j-1}) - V^{j-1}\| + h \bigl\|f(\theta_j,z(\theta_j))-f(\theta_j,V^j_\tau) \bigr\|
    \\ & \leq \int\limits_{t_{j-1}}^{\theta_j} \bigl\| z'(s)-z'(\theta_j)\bigl\|\,\mathrm{d}s + u_{j-1} + hL \| z(\theta_j)-V^j_\tau \| 
    \\ & \leq C_2h^{\varrho+1} + u_{j-1} + hL \| z(\theta_j)-V^j_\tau \|,
\end{align*}
where the last passage is by Lemma 2(ii) from \cite{randEuler} ($C_2$ is a constant from this lemma). Given that $hL\leq\frac12$, we have
$$\|z (\theta_{j})-V_\tau^j\| \leq 2C_2h^{\varrho+1} + 2u_{j-1}.$$
Let us insert this into \eqref{eq:semi:s2}:
\begin{equation*}
\max_{1\leq j\leq k}\|S_2^j\| \leq   2LC_2(b-a)h^{\varrho+1} + 2hL \sum_{j=0}^{k-1}u_{j}.    
\end{equation*}
Combining this with \eqref{eq:s1s2s3s4} yields
\begin{equation*}
    u_k \leq  \max_{1\leq j\leq k} \|S_1^j\|+2LC_2(b-a)h^{\varrho+1}+2hL \sum_{j=0}^{k-1} u_j
\end{equation*}
with probability $1$ for every $k\in\{1,\ldots,n\}$. By Gronwall's inequality,
\begin{equation} \label{eq:semi:bound_z_V}
    \max_{0\leq j\leq n} \|z(t_j)-V^j\| \leq \Bigl( \max_{1\leq j\leq n} \|S_1^j\|+2LC_2(b-a)h^{\varrho+1} \Bigr)\cdot e^{2L(b-a)}.
\end{equation}
By \eqref{eq:semi:bound_S1} and \eqref{eq:semi:bound_z_V}, we get \eqref{eq:theorem-error-1}.
\end{proof}

\begin{theorem} \label{theorem-error-2}
Let $p\in[2,\infty)$. There exists a constant $C = C(a,b,d,K,L,\varrho, p)>0$ such that for every $n\geq \lfloor b-a\rfloor+1$ satisfying the condition $Lh<1$ and for every $(\eta,f)\in F^{\varrho}$, it holds that
\begin{equation}
    \label{eq:theorem-error-2}
    \Bigl\| \max_{0\leq j \leq n} \bigl\| z(\eta,f)(t_j)- V^j_{\text{S}2}(\eta,f)\bigr\| \Bigr\|_p \leq C h^{\varrho+\frac12}.
\end{equation}
\end{theorem}
\begin{proof}
    We have
\begin{align}
    z(t_k)-V^k & = \sum_{j=1}^k \Bigl( \int\limits_{t_{j-1}}^{t_j} z'(s)\,\mathrm{d}s - hz'(\theta_j)\Bigr) \notag \\ & \ \ \ \ \ \ \ \ \  + h\sum_{j=1}^k \bigl[ f\bigl(\theta_j,z(\theta_j)\bigr) - f\bigl(\theta_j,(1-\tau_j)z(t_{j-1})+\tau_j z(t_j)\bigr) \bigr] \notag \\ & \ \ \ \ \ \ \ \ \  + h\sum_{j=1}^k \bigl[ f\bigl(\theta_j,(1-\tau_j)z(t_{j-1})+\tau_j z(t_j)\bigr) - f\bigl(\theta_j,(1-\tau_j)V^{j-1}+\tau_j V^j\bigr) \bigr] \notag \\ & = S_1^k + \tilde S_2^k + \tilde S_3^k. \label{eq:s1s2s3}
\end{align}

Note that the definition of $S_1^k$ is the same as in the proof of Theorem \ref{theorem-error-1} so the bound \eqref{eq:semi:bound_S1} remains valid.

Note that for every $k\in\{1,\ldots,n\}$,
\begin{align}
    \|\tilde S_2^k\| & \leq hL\sum_{j=1}^k \bigl\| z(\theta_j) - (1-\tau_j)z(t_{j-1}) - \tau_j z(t_j) \bigr\| \notag
    \\ & = hL\sum_{j=1}^k \Bigl\| \int\limits_{t_{j-1}}^{\theta_j}z'(s)\,\mathrm{d}s - \tau_jh\cdot \frac{z(t_j)-z(t_{j-1})}{t_j-t_{j-1}} \Bigr\| = hL\sum_{j=1}^k \Bigl\| \int\limits_{t_{j-1}}^{\theta_j}\bigl(z'(s)-z'(\xi_j)\bigr)\,\mathrm{d}s \Bigr\| \notag
     \\ & \leq hL\sum_{j=1}^k \int\limits_{t_{j-1}}^{\theta_j}\bigl\|z'(s)-z'(\xi_j)\bigr\|\,\mathrm{d}s \leq hLC \sum_{j=1}^k \int\limits_{t_{j-1}}^{\theta_j} |s-\xi_j |^\varrho\,\mathrm{d}s \notag \\ & \leq hLC \sum_{j=1}^n \int\limits_{0}^{h} s^\varrho\,\mathrm{d}s = \frac{LC(b-a)}{\varrho+1}h^{\varrho+1}. \label{eq:bound_S2}
\end{align}
In the first line, we used Assumption \ref{itm:A4}. In the second line, we used the mean-value theorem, $\xi_j\in (t_{j-1},t_j)$. In the third line, we used (43) from Lemma 2 in \cite{randEuler}.

Let us define $u_0=0$ and $$u_k = \max_{0\leq j\leq k} \|z(t_j)-V^j\|=  \max_{1\leq j\leq k} \|z(t_j)-V^j\|$$ for $k\in\{1,\ldots,n\}$. By Assumption \ref{itm:A4} we get
\begin{align*}
\max_{1\leq j\leq k}\|\tilde S_3^j\| & \leq hL \max_{1\leq j\leq k} \sum_{i=1}^j \bigl\| (1-\tau_i)(z(t_{i-1})-V^{i-1})+\tau_j (z(t_i)-V^i)\bigr\| \\ & \leq hL \sum_{j=1}^k\bigl[(1-\tau_j)\|z(t_{j-1})-V^{j-1}\|+\tau_j \|z(t_j)-V^j\|\bigr] \\ & \leq hL \sum_{j=1}^k \|z(t_{j-1})-V^{j-1}\| + hL \|z(t_k)-V^k\| \\ & \leq hL\sum_{j=0}^{k-1} u_j + hL u_k.
\end{align*}
Combining this with \eqref{eq:s1s2s3} yields
\begin{equation*}
    (1-hL)u_k \leq  \max_{1\leq j\leq k} \|S_1^j\|+ \max_{1\leq j\leq k}\|\tilde S_2^j\|+hL \sum_{j=0}^{k-1} u_j
\end{equation*}
with probability $1$ for every $k\in\{1,\ldots,n\}$. By Gronwall's inequality and $hL<\frac12$,
\begin{equation} \label{eq:bound_z_V}
    \max_{0\leq j\leq n} \|z(t_j)-V^j\| \leq \Bigl( 2\max_{1\leq j\leq n} \|S_1^j\|+ 2\max_{1\leq j\leq n}\|\tilde S_2^j\| \Bigr)\cdot e^{2L(b-a)}.
\end{equation}
By \eqref{eq:semi:bound_S1}, \eqref{eq:bound_S2}, and \eqref{eq:bound_z_V}, we obtain \eqref{eq:theorem-error-2}.
\end{proof}

\section{Probabilistic stability} \label{sec:stability}

\subsection{Test equation}

We consider Dahlquist's test equation \cite{dahlquist}:
\begin{equation}
	\label{TEST_PROBLEM}
 \tag{D}
		\left\{ \begin{array}{ll}
			z'(t)= \lambda z(t), \ t\geq 0, \\
			z(0) = 1
		\end{array}\right.
\end{equation}
with $\lambda\in\mathbb{C}$. 

For deterministic schemes, the region of absolute stability is defined as the region of such $h\lambda$ in the complex plane that the approximated solution to the test problem (with step size $h>0$) converges to $0$ in infinity. The method is called A-stable, if its region of absolute stability covers the left complex half-plane $\mathbb{C}_- = \{z\in\mathbb{C}\colon \Re(z)<0\}$. Of course, it is related to the fact that the exact solution of \eqref{TEST_PROBLEM} is $z(t)=\exp(\lambda t)$ and we have 
\begin{equation*}
    \lim\limits_{t\to\infty} z(t)=0 \ \hbox{iff} \ \Re(\lambda)<0.
\end{equation*}

Let $(V^k)_{k=0}^\infty$ be a point-wise approximation to the solution of \eqref{TEST_PROBLEM} obtained for a fixed step size $h>0$ by one of the schemes \eqref{eq:S1} or \eqref{eq:S2}. For both schemes,
\begin{equation} \label{eq:test_V1}
    V_{\text{S}0}^k =V_{\text{S}1}^k = \prod_{j=1}^k \Bigl( 1+\frac{z}{1-z\tau_j} \Bigr),
\end{equation}
where $z=\lambda h$. Hence, both considered schemes will have the same stability regions.

Similarly, as in \cite{taylor, randRK}, we consider the following regions of probabilistic stability:
\begin{eqnarray}
    &&\mathcal{R}_{MS} = \{z \in\mathbb{C} \colon V^k \to 0 \ \hbox{in $L^2(\Omega)$ as} \ k\to \infty \}, \label{eq:regMS} \\
    &&\mathcal{R}_{AS} = \{z \in\mathbb{C} \colon V^k \to 0 \  \hbox{almost surely as} \ k\to \infty\},\label{eq:regAS} \\
     &&\mathcal{R}_{SP} = \left\{z \in \mathbb{C} \colon  V^k \to 0 \text{ in probability as} \ k\to \infty  \right\},
    \label{eq:regSP}
\end{eqnarray}
where $z=\lambda h$. The region $\mathcal{R}_{MS}$ is called the region of mean-square stability, $\mathcal{R}_{AS}$ -- the region of asymptotic stability, and $\mathcal{R}_{SP}$ -- the region of stability in probability. 

We consider also an absolute stability region of deterministic schemes \eqref{eq:detS1} and \eqref{eq:detS2} as a reference region. These two deterministic schemes are A-stable:
\begin{equation}
    \mathcal{R}_{ref} =  \Bigl\{ z\in\mathbb{C} \colon z\neq 2 \wedge \Bigl| 1+\frac{z}{1-\frac12 z} \Bigr| <1 \Bigr\} = \bigl\{ z\in\mathbb{C}\colon |z+2| < |z-2| \bigr\} = \mathbb{C}_-. \label{eq:ref1} 
\end{equation}

\subsection{Regions of mean-square stability}

In this section, we prove the basic properties of $\mathcal{R}_{MS}$. In particular, we show that this region is bounded (Fact \ref{f:MS:bounded}), which implies that the randomized implicit RK2 scheme is not A-stable in the mean-square probability sense.

Let us start by rewriting the definition \eqref{eq:regMS} in a more convenient form. It is easy to see that 
$$\mathbb{E}|V_{\text{S}1}^k|^2 = \prod_{j=1}^k \mathbb{E}\Bigl| 1+\frac{z}{1-z\tau_j} \Bigr|^2 = \Bigl[\mathbb{E}\Bigl| 1+\frac{z}{1-z\tau_1} \Bigr|^2\Bigr]^k,$$
cf. \eqref{eq:test_V1}. Thus, 
\begin{equation} \label{eq:regMS2}
    \mathcal{R}_{MS} = \Bigl\{z \in\mathbb{C} \colon  \mathbb{E}\Bigl| 1+\frac{z}{1-z\tau_1} \Bigr|^2 < 1 \Bigr\}.
\end{equation}

\begin{fact} \label{f:MS:conjugate}
    The region $\mathcal{R}_{MS}$ is symmetric with respect to the real axis.
\end{fact}
\begin{proof}
    Note that for each $t\in [0,1]$ such that $tz\neq 1$ we have 
    $$\Bigl| 1 + \frac{z}{1-tz} \Bigr| = \left| \overline{\Bigl(\frac{1-tz+z}{1-tz}\Bigr)} \right|=\left| \frac{\overline{1-tz+z}}{\overline{1-tz}} \right| = \Bigl| \frac{1-t\bar z+\bar z}{1-t\bar z} \Bigr| = \Bigl| 1 + \frac{\bar z}{1-t\bar z} \Bigr|.$$
\end{proof}
    
\begin{fact} \label{f:MS:C-}
    $\mathcal{R}_{MS}\subset \mathbb{C}_-$.
\end{fact}
\begin{proof}
    Let $z=a+bi$, $a\geq 0$ and $b\in\mathbb{R}$. We will show that $\displaystyle \mathbb{E}\Bigl| 1+\frac{z}{1-z\tau_1} \Bigr|^2\geq 1$, which is equivalent to $z\notin \mathcal{R}_{MS}$. By Fact \ref{f:MS:conjugate}, it suffices to consider $b\geq 0$.

    \textit{Case 1:} $a\in [0,1)$ and $b=0$. Then $\displaystyle \frac{a}{1-ta}\geq 0$ for all $t\in [0,1]$ and as a result,
    $$\mathbb{E}\Bigl| 1+\frac{z}{1-z\tau_1} \Bigr|^2 = \Bigl(1+ \frac{a}{1-\tau_1 a} \Bigr)^2\geq 1.$$

    \textit{Case 2:} $a\in [1,\infty)$ and $b=0$. Then $1-ta=0$ for $t=\frac{1}{a}$ and
    \begin{align*}
        \mathbb{E}\Bigl| 1+\frac{z}{1-z\tau_1} \Bigr|^2 & \geq \int\limits_0^{1/a} \Bigl(1+\frac{a}{1-ta}\Bigr)^2\,\mathrm{d}t = \int\limits_0^{1/a} \Bigl(1+\frac{2a}{1-ta}+\frac{a^2}{(1-ta)^2}\Bigr)\,\mathrm{d}t \\ & = \frac{1}{a} - a + \lim_{t\to \frac{1}{a}^-} \Bigl( \frac{a}{1-ta}-2\log (1-ta)\Bigr) = \infty.
    \end{align*}

    \textit{Case 3:} $a\geq 0$ and $b>0$. Then $tz\neq 1$ for all $t\in [0,1]$. Let us put
    $$I = \mathbb{E}\Bigl| 1+\frac{z}{1-z\tau_1} \Bigr|^2 = \int\limits_0^1 \Bigl|\frac{1-tz+z}{1-tz}\Bigr|^2\,\mathrm{d}t = \int\limits_0^1 \Bigl|\frac{1+sz}{1-z+sz}\Bigr|^2\,\mathrm{d}s,$$
    where we substituted $s=1-t$. Hence,
    \begin{align}
        I & = \int\limits_0^1 \frac12 \cdot \Bigl( \Bigl|\frac{1-tz+z}{1-tz}\Bigr|^2 + \Bigl|\frac{1+tz}{1-z+tz}\Bigr|^2\Bigr)\,\mathrm{d}t \geq \int\limits_0^1  \Bigl|\frac{1-tz+z}{1-tz}\Bigr|\cdot \Bigl|\frac{1+tz}{1-z+tz}\Bigr|\,\mathrm{d}t \notag \\ & = \int\limits_0^1  \frac{|1-(tz-z)|}{|1-(z-tz)|}\cdot \frac{|1+tz|}{|1-tz|} \,\mathrm{d}t \label{eq:MS:C-}
    \end{align}

    \begin{figure}[h]
        \centering
        \begin{subfigure}{0.48\textwidth}
            \centering
            \begin{tikzpicture}
                \draw[line width=0.25mm, red ] (0.4,1) -- (3,0);
                \draw[line width=0.25mm, red ] (-0.4,-1) -- (3,0);
                \draw (0,0) node[anchor=east]{$0$}
                  -- (3,0) node[anchor=west]{$1$}
                  -- (4,2.5) node[anchor=west]{$1+z$}
                  -- (1,2.5) node[anchor=east]{$z$}
                  -- (-1,-2.5) node[anchor=east]{$-z$}
                  -- (2,-2.5) node[anchor=west]{$1-z$};
                \draw (2,-2.5) -- (3,0);
                \draw[fill=black] (0,0) circle(2pt);
                \draw[fill=black] (1,2.5) circle(2pt);
                \draw[fill=black] (-1,-2.5) circle(2pt);
                \draw[fill=black] (3,0) circle(2pt);
                \draw[fill=black] (2,-2.5) circle(2pt);
                \draw[fill=black] (4,2.5) circle(2pt);
                \draw (0,0) [black,domain=0:68.2] plot ({0.75*cos(\x)}, {0.75*sin(\x)});
                \draw (0,0) node[anchor=215,inner sep=6.5pt]{$\alpha$};
                \draw[fill=black] (0.4,1) circle(2pt);
                \draw[fill=black] (-0.4,-1) circle(2pt);
                \draw (0.4,1) node[anchor=east]{$z-tz$};
                \draw (-0.4,-1) node[anchor=east]{$tz-z$};
                \draw (0.4+2.6*0.3,1-0.3) node[anchor=210, red, rotate=-20.5]{$|1-(z-tz)|$};
                \draw (-0.4+3.4*0.3,-1+0.3) node[anchor=150, red, rotate=17]{$|1-(tz-z)|$};
            \end{tikzpicture}
        \end{subfigure}
        \begin{subfigure}{0.48\textwidth}
            \centering
            \begin{tikzpicture}
                \draw[line width=0.25mm, red ] (-0.6,-1.5) -- (3,0);
                \draw[line width=0.25mm, red ] (0.6,1.5) -- (3,0);
                \draw (0,0) node[anchor=east]{$0$}
                  -- (3,0) node[anchor=west]{$1$}
                  -- (4,2.5) node[anchor=west]{$1+z$}
                  -- (1,2.5) node[anchor=east]{$z$}
                  -- (-1,-2.5) node[anchor=east]{$-z$}
                  -- (2,-2.5) node[anchor=west]{$1-z$};
                \draw (2,-2.5) -- (3,0);
                \draw[fill=black] (0,0) circle(2pt);
                \draw[fill=black] (1,2.5) circle(2pt);
                \draw[fill=black] (-1,-2.5) circle(2pt);
                \draw[fill=black] (3,0) circle(2pt);
                \draw[fill=black] (2,-2.5) circle(2pt);
                \draw[fill=black] (4,2.5) circle(2pt);
                \draw (0,0) [black,domain=0:68.2] plot ({0.75*cos(\x)}, {0.75*sin(\x)});
                \draw (0,0) node[anchor=215,inner sep=6.5pt]{$\alpha$};
                \draw[fill=black] (0.6,1.5) circle(2pt);
                \draw[fill=black] (-0.6,-1.5) circle(2pt);
                \draw (0.6,1.5) node[anchor=east]{$tz$};
                \draw (-0.6,-1.5) node[anchor=east]{$-tz$};
                \draw (0.6+2.4*0.3,1.5-1.5*0.3) node[anchor=210, red, rotate=-30]{$|1-tz|$};
                \draw (-0.6+3.6*0.3,-1.5+1.5*0.3) node[anchor=150, red, rotate=22]{$|1+tz|$};
            \end{tikzpicture}
        \end{subfigure}
        \caption{Geometric interpretation of moduli that appeared in \eqref{eq:MS:C-}.}
        \label{fig:MS:C-}
    \end{figure}

    Moduli $|1-(tz-z)|$, $|1-(z-tz)|$, $|1+tz|$, $|1-tz|$ (for any $t\in [0,1]$) are visualized as segments on the complex plane in Figure \ref{fig:MS:C-}. By the conditions $a\geq 0$ and $b>0$, we know that $\alpha \in \bigl(0,\frac{\pi}{4}\bigr)$, where $\alpha$ is an angle marked in Figure \ref{fig:MS:C-}. Thus, $|1-(tz-z)|\geq |1-(z-tz)|$ and $|1+tz|\geq |1-tz|$. Consequently, $I\geq 1$, which concludes the proof.
\end{proof}

\begin{fact} \label{f:MS:bounded}
    $\mathcal{R}_{MS}$ is bounded.
\end{fact}
\begin{proof}
    Let $z=a+bi$, where $a<0$ (see Fact \ref{f:MS:C-}) and $b\in\mathbb{R}$. We will show that for sufficiently big $|z|$, we have $z\notin \mathcal{R}_{MS}$, that is $\displaystyle \mathbb{E}\Bigl| 1+\frac{z}{1-z\tau_1} \Bigr|^2\geq 1$.

    We have
    \begin{equation}
        \label{eq:MS:bounded:1}
        \mathbb{E}\Bigl| 1+\frac{z}{1-z\tau_1} \Bigr|^2 \geq \Bigl| 1+z\mathbb{E}\Bigl(\frac{1}{1-z\tau_1}\Bigr) \Bigr|^2 \geq \Bigl[ 1 + \Re\Bigl( z\mathbb{E}\Bigl(\frac{1}{1-z\tau_1} \Bigr) \Bigr]^2 
    \end{equation}    
    Note that
    $$\mathbb{E}\Bigl(\frac{1}{1-z\tau_1}\Bigr) = \int\limits_0^1\frac{\mathrm{d}t}{1-at-bti} = \int\limits_0^1\frac{1-at}{(1-at)^2+(bt)^2}\,\mathrm{d}t + i \int\limits_0^1\frac{bt}{(1-at)^2+(bt)^2}\,\mathrm{d}t$$
    and as a result,
    \begin{equation*}
         \Re\Bigl( z\mathbb{E}\Bigl(\frac{1}{1-z\tau_1} \Bigr) \Bigr) = \int\limits_0^1\frac{a(1-at)-b\cdot bt}{(1-at)^2+(bt)^2}\,\mathrm{d}t = \int\limits_0^1\frac{a-|z|^2t}{1-2at+|z|^2t^2}\,\mathrm{d}t = -\frac12 \log\bigl( |z|^2-2a+1 \bigr).
    \end{equation*}
    Substituting this to \eqref{eq:MS:bounded:1}, we get
    $$\mathbb{E}\Bigl| 1+\frac{z}{1-z\tau_1} \Bigr|^2 \geq \Bigl(1-\frac12 \log\bigl( |z|^2-2a+1 \bigr)\Bigr)^2.$$
    Consider $|z|\geq \sqrt{e^4-1}$. Then by $a<0$ and $|z|^2+1\geq e^4$, we have
    $$\mathbb{E}\Bigl| 1+\frac{z}{1-z\tau_1} \Bigr|^2 \geq \Bigl(\frac12 \log\bigl( |z|^2-2a+1 \bigr)-1\Bigr)^2 > \Bigl(\frac12 \log\bigl( |z|^2+1 \bigr)-1\Bigr)^2\geq 1.$$
\end{proof}

\begin{remark}
    By Fact \ref{f:MS:C-} and Fact \eqref{f:MS:bounded}, $\mathcal{R}_{MS}$ is a bounded subset of $\mathbb{C}_-$. By \eqref{eq:ref1}, $\mathcal{R}_{ref}=\mathbb{C}_-$. Hence, randomized schemes \eqref{eq:S1} and \eqref{eq:S2} are not A-stable in the mean-square stability sense, although their deterministic counterparts (obtained by putting $\tau_j=\frac12$ for all $j\in\mathbb{Z}_+$) are A-stable.  
\end{remark}

\noindent In the next fact, we consider the interval of mean-square stability, defined as
\begin{equation*}
    \mathcal{I}_{MS} = \mathcal{R}_{MS}\cap \mathbb{R}.
\end{equation*}

\begin{fact} \label{f:MS:interval}
    $\displaystyle \mathcal{I}_{MS}=\Bigl\{ a\in (-\infty, 0) \colon \mathbb{E}\Bigl(1+\frac{a}{1- a\tau_1}\Bigr)^2 < 1 \Bigr\} = (-x_0,0)$, where $4.03<x_0<4.04$.
\end{fact}
\begin{proof}
    By Fact \ref{f:MS:C-}, $\mathcal{I}_{MS} \subset (-\infty, 0)$. Note that for $a<0$,
    \begin{align}
        \mathbb{E}\Bigl(1+\frac{a}{1- a\tau_1}\Bigr)^2 & = \int\limits_0^1 \Bigl(1+\frac{2a}{1+ta}+\frac{a^2}{(1-ta)^2}\Bigr)\,\mathrm{d}t = \Bigl( t-2\log (1-ta) + \frac{a}{1-ta} \Bigr)\Bigr|_0^1 \notag \\ & = 1 +\frac{a^2}{1-a} - 2\log (1-a) . \label{eq:MS:interval}
    \end{align}
    Hence, the condition $\displaystyle \mathbb{E}\Bigl(1+\frac{a}{1- a\tau_1}\Bigr)^2 < 1$ is equivalent to $\displaystyle \frac{a^2}{1-a} - 2\log (1-a)< 0$. Consider a function $$g\colon (-\infty,0)\ni a \mapsto \frac{a^2}{1-a} - 2\log (1-a)\in\mathbb{R}.$$
    Since $\displaystyle g'(a)=\frac{2-a^2}{(1-a)^2}$, we conclude that $g$ decreases in $\bigl(-\infty,-\sqrt{2}\bigr)$ and increases in $\bigl(-\sqrt{2},0\bigr)$. We have also $g(0)=0$, $g(-4.03)\approx-0.002<0$, and $g(-4.04)\approx 0.0036 >0$. This concludes the proof.
\end{proof}

\noindent In the following fact, we express the expected value from \eqref{eq:regMS2} in an explicit form. This is useful when drawing $\mathcal{R}_{MS}$, cf. Figure \ref{fig:regions}. This fact will be also used to show that the mean-square stability region is open (Fact \ref{f:MS:open}).

\begin{fact} \label{f:MS:formula}
    When $z=a+bi$, $a<0$ and $b\in\mathbb{R}\setminus \{0\}$, then $$\mathbb{E}\Bigl| 1+\frac{z}{1-z\tau_1} \Bigr|^2 = 1-\log (1-2a+a^2+b^2)+\frac{a^2+b^2}{|b|}\cdot \arctan \Bigl( \frac{|b|}{1-a} \Bigr).$$
\end{fact}
\begin{proof}
    Note that for any $t\in [0,1]$,
    $$\Bigl|1+\frac{z}{1-tz}\Bigr|^2 = \frac{|1-at+a+(b-bt)i|^2}{|1-at+bti|^2} = 1+\frac{2a+|z|^2(1-2t)}{1-2at+|z|^2t^2}.$$
    Thus,
    \begin{align*}
        \mathbb{E}\Bigl| 1+\frac{z}{1-z\tau_1} \Bigr|^2 & = \int\limits_0^1 \Bigl( 1+\frac{2a+|z|^2(1-2t)}{1-2at+|z|^2t^2} \Bigr)\,\mathrm{d}t 
        \\ & = 1-\int\limits_0^1 \frac{\frac{\mathrm{d}}{\mathrm{d}t}\bigl( 1-2at+|z|^2t^2 \bigr)}{1-2at+|z|^2t^2}\,\mathrm{d}t + \int\limits_0^1 \frac{\frac{|z|^4}{b^2}}{1+\left(\frac{|z|^2}{|b|}t-\frac{a}{|b|}\right)^2}\,\mathrm{d}t \\ & = 1 -\log(1-2at+|z|^2t^2)\Bigr|_0^1 + \frac{|z|^2}{|b|}\arctan\Bigl(  \frac{|z|^2}{|b|}t-\frac{a}{|b|} \Bigr)\Bigr|_0^1 \\ & = 1 -\log(1-2a+|z|^2) + \frac{|z|^2}{|b|}\cdot\Bigl[\arctan\Bigl(  \frac{|z|^2}{|b|}-\frac{a}{|b|} \Bigr)-\arctan\Bigl( -\frac{a}{|b|} \Bigr)\Bigr] \\ & = 1-\log (1-2a+a^2+b^2)+\frac{a^2+b^2}{|b|}\cdot \arctan \Bigl( \frac{|b|}{1-a} \Bigr),
    \end{align*}
    where in the last passage we used the identity
    $$\arctan x - \arctan y = \arctan \frac{x-y}{1+xy}$$
    which holds for any $x, y\in\mathbb{R}$ such that $\arctan x - \arctan y \in \bigl(-\frac{\pi}{2},\frac{\pi}{2}\bigr)$.
\end{proof}

\begin{fact} \label{f:MS:open}
    $\mathcal{R}_{MS}$ is an open subset of $\mathbb{C}$.
\end{fact}
\begin{proof}
    We use Fact \ref{f:MS:formula}. Note that when $a<0$ is fixed and $b\to 0$,
    $$1-\log (1-2a+a^2+b^2)+\frac{a^2+b^2}{|b|}\cdot \arctan \Bigl( \frac{|b|}{1-a} \Bigr) \to 1 +\frac{a^2}{1-a} - 2\log (1-a),$$
    cf. \eqref{eq:MS:interval}. Hence, the function $\displaystyle h\colon \mathbb{C}_-\ni z \mapsto \mathbb{E}\Bigl| 1+\frac{z}{1-z\tau_1} \Bigr|^2 \in \mathbb{R}$ is continuous and $\mathcal{R}_{MS}=h^{-1}\bigl((-\infty,1)\bigr)\subset \mathbb{C}_-$ (cf. Fact \ref{f:MS:C-}) is open.
\end{proof}

The mean-square stability region of the randomized implicit RK2 scheme is shown in Figure \ref{fig:regions} along with stability regions (mean-square and asymptotic) for the randomized explicit RK2 scheme \eqref{eq:explRK2} and absolute stability region of the deterministic midpoint scheme \eqref{eq:detRK2}. Note that the absolute stability region for deterministic implicit schemes \eqref{eq:detS1} and \eqref{eq:detS2}, not displayed in Figure \ref{fig:regions}, covers the entire left half-plane. The asymptotic stability and stability in probability regions of the randomized implicit RK2 scheme will be analyzed in section \ref{sec:AS}. As we will show, these regions also coincide with the left half-plane.

\begin{figure}
    \centering
    \includegraphics[scale=0.8]{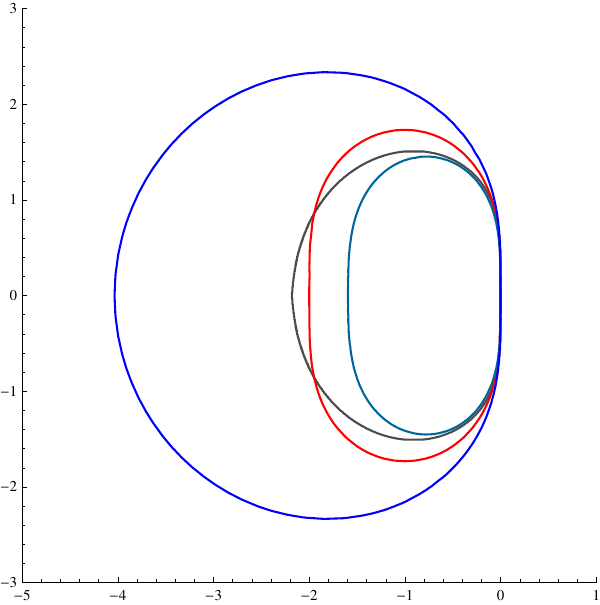}
    \caption{Contours of the stability regions for the RK2 scheme: \textcolor{red-new}{deterministic explicit}, \textcolor{gray-new}{randomized explicit (asymptotic) $\mathcal{R}^e_{AS}$}, \textcolor{light-blue}{randomized explicit (mean-square) $\mathcal{R}^e_{MS}$}, and \textcolor{dark-blue}{randomized implicit (mean-square) $\mathcal{R}_{MS}$}.}
    \label{fig:regions}
\end{figure}

\subsection{Regions of asymptotic stability and stability in probability} \label{sec:AS}

In this section, we show that the randomized implicit and semi-implicit RK2 schemes are asymptotically A-stable and A-stable in probability.

\begin{fact}
    \label{f:AS=SP}
    For schemes \eqref{eq:S1} and \eqref{eq:S2}, the notions of asymptotic stability and stability in probability are equivalent. Moreover, 
    \begin{equation}
    \label{eq:as-sp-1}
        \mathcal{R}_{AS} = \mathcal{R}_{SP} = \Bigl\{ z\in\mathbb{C}\colon \mathbb{E}\Bigl( \log \Bigl|  1+\frac{z}{1-z\tau_1} \Bigr| \Bigr) < 0 \Bigr\} = \mathbb{C}_-.
    \end{equation}
\end{fact}
\begin{proof}
    The initial two equalities in \eqref{eq:as-sp-1} follow from \eqref{eq:regAS}, \eqref{eq:regSP}, Lemma 5.1 in \cite{higham2}, and Lemma 4 in \cite{randRK}. To use these lemmas, it suffices to check whether 
    \begin{align}
        \mathbb{E}\Bigl( \log \Bigl|  1+\frac{z}{1-z\tau_1} \Bigr| \Bigr)^2 & = \int\limits_0^1 \bigl( \log |1+z-tz| - \log |1-tz| \bigr)^2\,\mathrm{d}t \notag \\ & \leq 2 \int\limits_0^1 \bigl( \log |1+z-tz| \bigr)^2\,\mathrm{d}t + 2 \int\limits_0^1 \bigl( \log |1-tz| \bigr)^2\,\mathrm{d}t    \label{eq:ln-sq-1}
    \end{align}
    is finite. Note that $t\mapsto \log |1+z-tz|$ is either continuous for $t\in [0,1]$ or has exactly one singularity; the latter holds iff $z\in (-\infty,-1]$. Similarly, $t\mapsto \log |1-tz|$ is either continuous for $t\in [0,1]$ or has exactly one singularity; the latter holds iff $z\in [1,\infty)$. Moreover, it is known that 
    \begin{equation*}
        \int\limits_{-\varepsilon}^\varepsilon (\log |t|)^2\,\mathrm{d}t < \infty
    \end{equation*}
    for any $\varepsilon>0$. Thus, expression in \eqref{eq:ln-sq-1} is finite.

    Now we will show the last equality in \eqref{eq:as-sp-1}. Note that
    \begin{align*}
        \mathbb{E}\Bigl( \log \Bigl|  1+\frac{z}{1-z\tau_1} \Bigr|\Bigl) & = \int\limits_0^1 \log | 1+z(1-t)|\,\mathrm{d}t -\int\limits_0^1 \log | 1-zt |\,\mathrm{d}t \\ &  = \int\limits_0^1 \log | 1+zt|\,\mathrm{d}t -\int\limits_0^1 \log | 1-zt |\,\mathrm{d}t \\ & =  \int\limits_0^1 \log \sqrt{1+2\Re(z)t+|z|^2t^2}\,\mathrm{d}t -\int\limits_0^1 \log \sqrt{1-2\Re(z)t+|z|^2t^2}\,\mathrm{d}t.
    \end{align*}
    Thus, $\displaystyle \mathbb{E}\Bigl( \log \Bigl|  1+\frac{z}{1-z\tau_1} \Bigr| \Bigr) < 0$ if and only if $\Re(z)<0$. This completes the proof.
\end{proof}

\begin{remark} \label{rem:A-stab}
    Fact \ref{f:MS:bounded} and the last equality in \eqref{eq:as-sp-1} mean that schemes \eqref{eq:S1} and \eqref{eq:S2} are A-stable in the asymptotic stability sense and in the stability in probability sense but not in the mean square stability sense.
\end{remark}

\section{Numerical experiments} \label{sec:experiments}

This section is devoted to an empirical study of the stability of schemes \eqref{eq:S1} and \eqref{eq:S2}. The performance of both schemes will be compared to their deterministic counterparts \eqref{eq:detS1} and \eqref{eq:detS2}.

We consider the following test problem, brought in \cite{hairer} as an example of a stiff problem:
\begin{equation}
	\label{eq:num-test}
		\left\{ \begin{array}{ll}
			z'(t)= -50\cdot \bigl( z(t) - \cos t \bigr), \ t\geq 0, \\[2pt]
			z(0) = 1, 
		\end{array}\right. \tag{$T$}
\end{equation}
The exact solution of \eqref{eq:num-test} is
$$z(t) = \frac{1}{2501} \bigl( e^{-50t} + 2500\cdot \cos t + 50 \cdot \sin t \bigr).$$
This exact solution is displayed as a solid green line in Figures \ref{fig:stability-semiimpl} and \ref{fig:stability-impl}. Approximate solutions obtained by \eqref{eq:S1} and \eqref{eq:detS1} (in Figure \ref{fig:stability-semiimpl}) or by \eqref{eq:S2} and \eqref{eq:detS2} (in Figure \ref{fig:stability-impl}) are represented by the dashed black line. The red line represents the approximation error (the difference between exact and approximated solutions). All schemes are applied with three different step sizes: $h\in\bigl\{\frac12,\frac14,\frac18\bigr\}$. Since for randomized schemes the approximated solution is stochastic, for each considered $h$ we include three sample paths. We display the solution on the interval $[0,50]$ 

We run also deterministic and randomized explicit RK2 schemes for the test problem \eqref{eq:num-test}. The errors for these schemes exploded, reaching values of order $10^{200}$--$10^{300}$ on the interval $[0,50]$ for all $h\in\bigl\{\frac12,\frac14,\frac18\bigr\}$. As a result, charts for explicit schemes would not be informative and thus we do not include them.

We make the following observations and conclusions:
\begin{itemize}
    \item We have not observed meaningful differences in terms of stability between implicit and semi-implicit RK2 schemes (both in deterministic and randomized settings). This is consistent with the fact that stability regions of \eqref{eq:S1} and \eqref{eq:S2} coincide, similarly as stability regions of \eqref{eq:detS1} and \eqref{eq:detS2}. 
    \item Each of schemes \eqref{eq:S1}, \eqref{eq:S2}, \eqref{eq:detS1}, and \eqref{eq:detS2} significantly outperforms explicit schemes \eqref{eq:detRK2} and \eqref{eq:explRK2} in terms of stability. Deterministic implicit RK2 schemes \eqref{eq:detS1} and \eqref{eq:detS2} provide approximations matching the exact solution with good accuracy even for $h=\frac12$. For randomized implicit RK2 schemes, there are significant errors for $h=\frac12$ and $h=\frac14$ but their magnitude remains bounded. For $h=\frac18$, randomized implicit RK2 schemes perform comparably to their deterministic counterparts. At the same time, errors of explicit RK2 schemes (both deterministic and randomized) reach extremely high levels for all tested values of step size $h\in\bigl\{\frac12,\frac14,\frac18\bigr\}$ and they explode (they seem to be unbounded). Observations made in this bullet are consistent with theoretical results showing that stability regions for implicit schemes are larger than for explicit schemes (cf. \eqref{eq:ref1}, Figure \ref{fig:regions}, and Fact \ref{f:AS=SP}). 
    \item We suppose that there is a link between the observed behavior of approximated solutions obtained via randomized implicit RK2 schemes (i.e., significant but bounded errors with tendency to return to the proximity of $0$ after pikes) and the fact that these schemes are A-stable in the asymptotic sense but not A-stable in the mean-square sense (see Remark \ref{rem:A-stab}). Intuitively, asymptotic A-stability means that the approximated solution remains stable (i.e., the approximation error remains bounded) on almost every trajectory. This explains why approximated solutions in Figures \ref{fig:stability-semiimpl} and \ref{fig:stability-impl} do not explode. On the other hand, lack of mean-square A-stability may be related to local fluctuations of the error. However, currently this is only a conjecture that requires further investigation based on a more formal approach. 
\end{itemize}

 \begin{figure}[h]
    \centering
    \begin{subfigure}{0.3\textwidth}
    \centering
        \includegraphics[width=.95\linewidth]{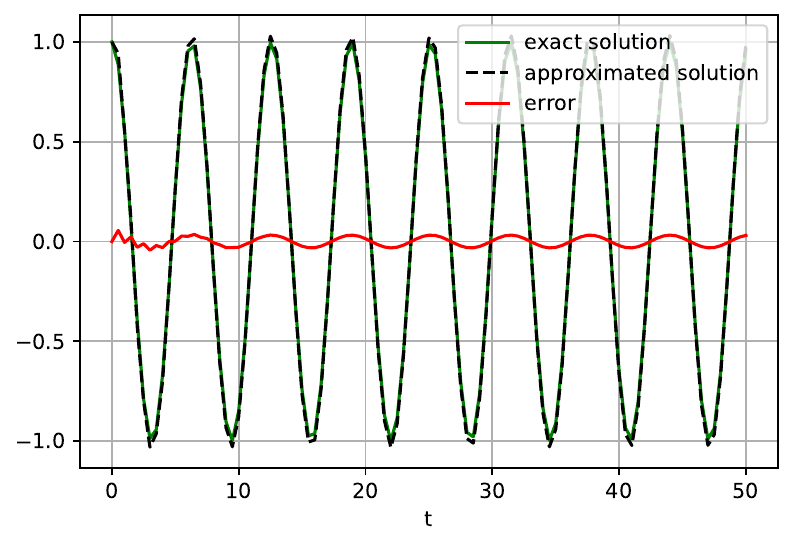}
        \subcaption{\eqref{eq:detS1}, $h=\frac12$}
    \end{subfigure}
    \begin{subfigure}{0.04\textwidth}
    \end{subfigure}
    \begin{subfigure}{0.3\textwidth}
    \centering
        \includegraphics[width=.95\linewidth]{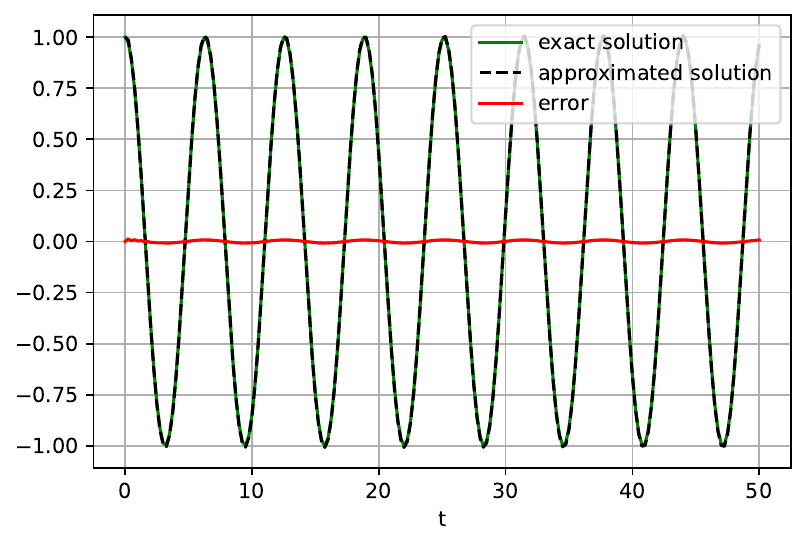}
        \subcaption{\eqref{eq:detS1}, $h=\frac14$}
    \end{subfigure}
    \begin{subfigure}{0.04\textwidth}
    \end{subfigure}
    \begin{subfigure}{0.3\textwidth}
    \centering
        \includegraphics[width=.95\linewidth]{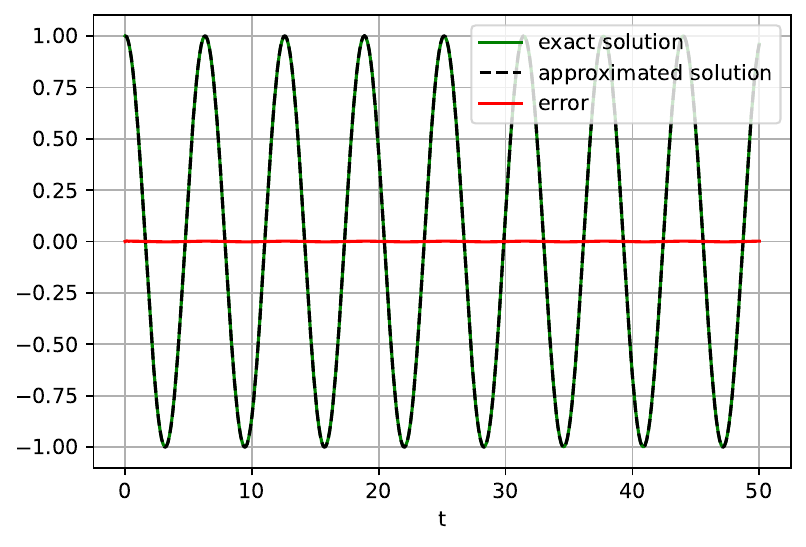}
        \subcaption{\eqref{eq:detS1}, $h=\frac18$}
    \end{subfigure}
    
    \begin{subfigure}{0.3\textwidth}
    \centering
        \includegraphics[width=.95\linewidth]{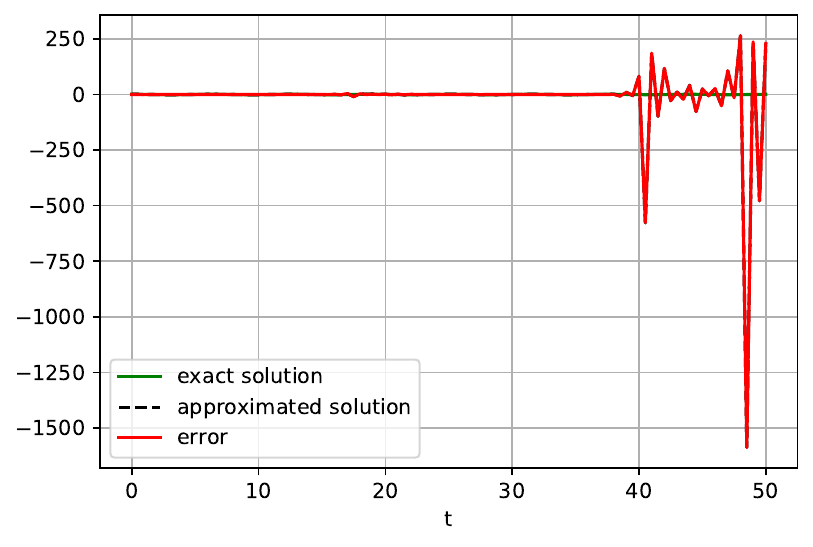}
        \subcaption{\eqref{eq:S1}, $h=\frac12$, path 1}
    \end{subfigure}
    \begin{subfigure}{0.04\textwidth}
    \end{subfigure}
    \begin{subfigure}{0.3\textwidth}
    \centering
        \includegraphics[width=.95\linewidth]{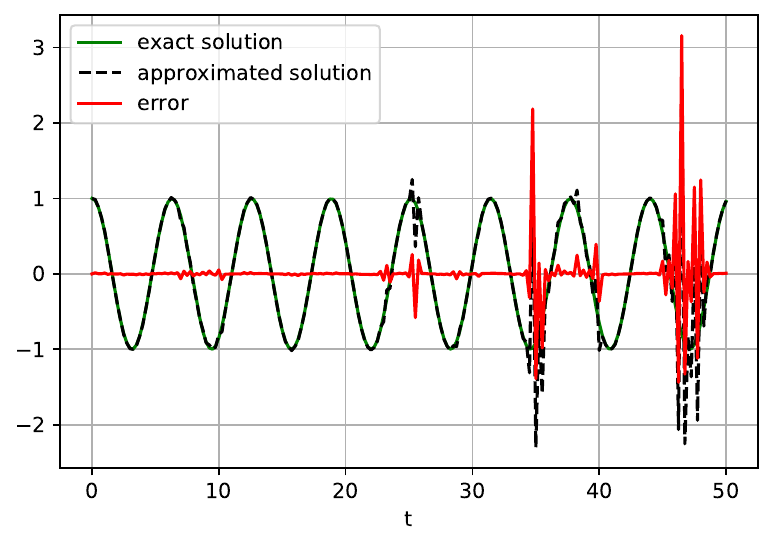}
        \subcaption{\eqref{eq:S1}, $h=\frac14$, path 1}
    \end{subfigure}
    \begin{subfigure}{0.04\textwidth}
    \end{subfigure}
    \begin{subfigure}{0.3\textwidth}
    \centering
        \includegraphics[width=.95\linewidth]{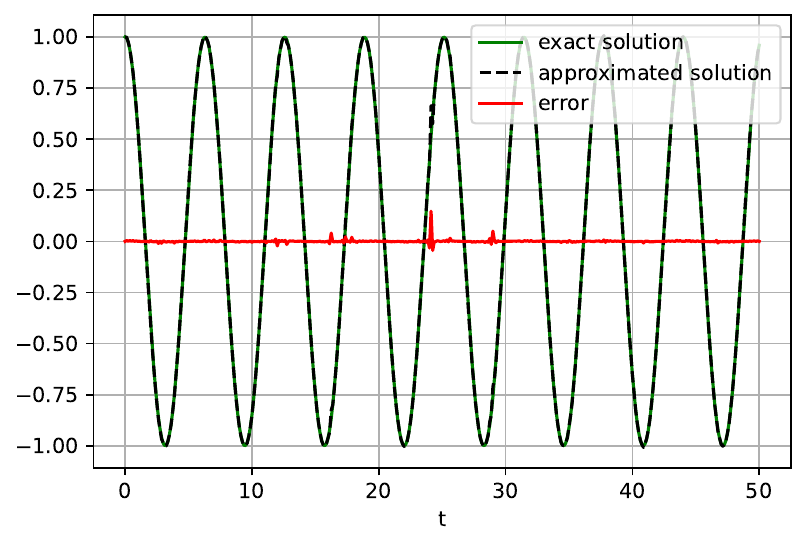}
        \subcaption{\eqref{eq:S1}, $h=\frac18$, path 1}
    \end{subfigure}
    
    \begin{subfigure}{0.3\textwidth}
    \centering
        \includegraphics[width=.95\linewidth]{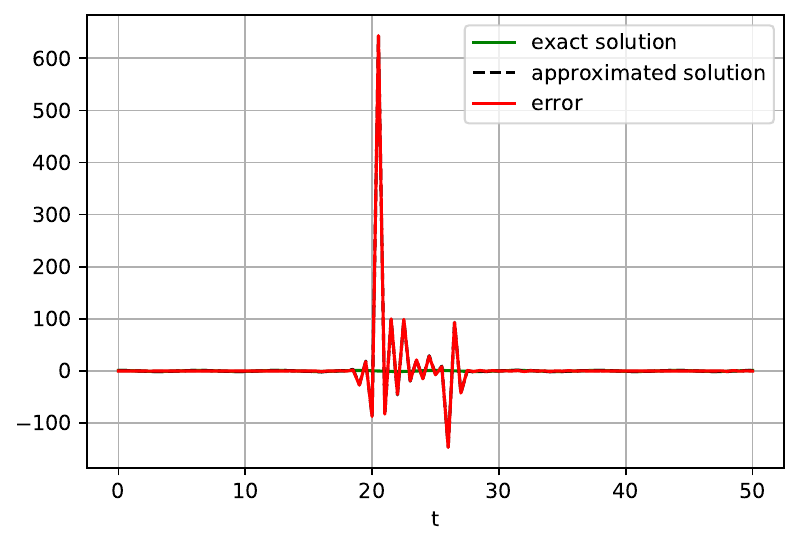}
        \subcaption{\eqref{eq:S1}, $h=\frac12$, path 2}
    \end{subfigure}
    \begin{subfigure}{0.04\textwidth}
    \end{subfigure}
    \begin{subfigure}{0.3\textwidth}
    \centering
        \includegraphics[width=.95\linewidth]{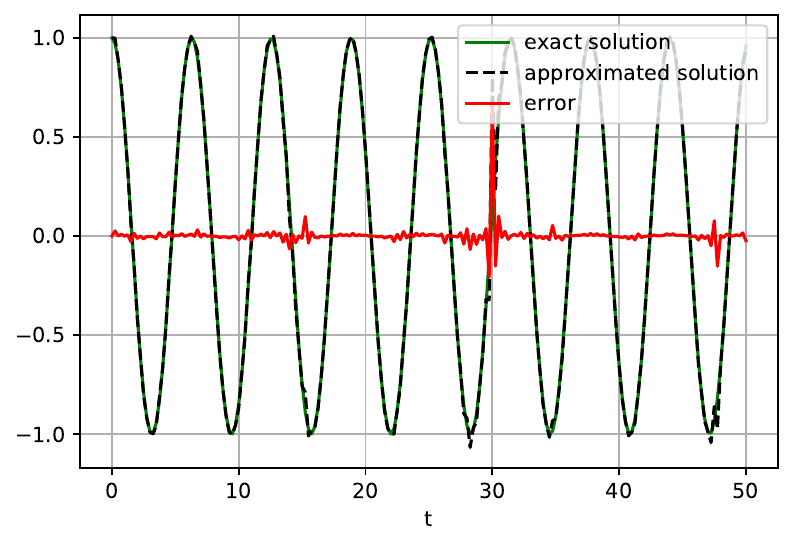}
        \subcaption{\eqref{eq:S1}, $h=\frac14$, path 2}
    \end{subfigure}
    \begin{subfigure}{0.04\textwidth}
    \end{subfigure}
    \begin{subfigure}{0.3\textwidth}
    \centering
        \includegraphics[width=.95\linewidth]{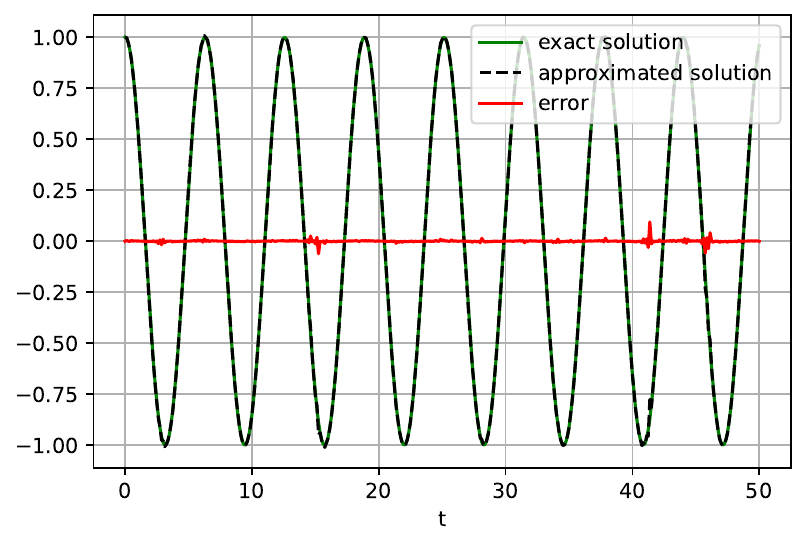}
        \subcaption{\eqref{eq:S1}, $h=\frac18$, path 2}
    \end{subfigure}
    
    \begin{subfigure}{0.3\textwidth}
    \centering
        \includegraphics[width=.95\linewidth]{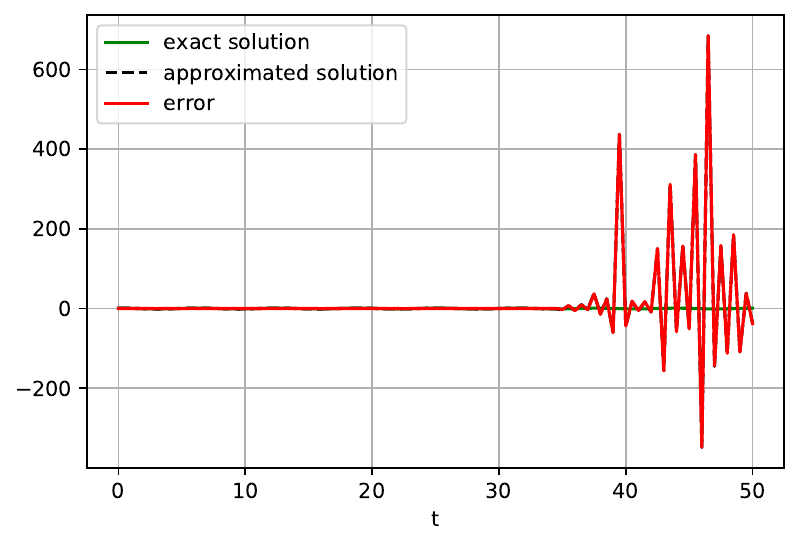}
        \subcaption{\eqref{eq:S1}, $h=\frac12$, path 3}
    \end{subfigure}
    \begin{subfigure}{0.04\textwidth}
    \end{subfigure}
    \begin{subfigure}{0.3\textwidth}
    \centering
        \includegraphics[width=.95\linewidth]{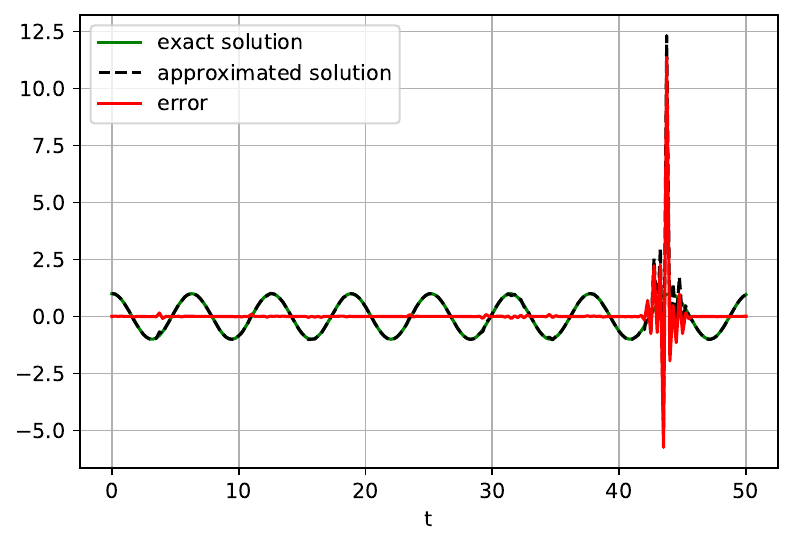}
        \subcaption{\eqref{eq:S1}, $h=\frac14$, path 3}
    \end{subfigure}
    \begin{subfigure}{0.04\textwidth}
    \end{subfigure}
    \begin{subfigure}{0.3\textwidth}
    \centering
        \includegraphics[width=.95\linewidth]{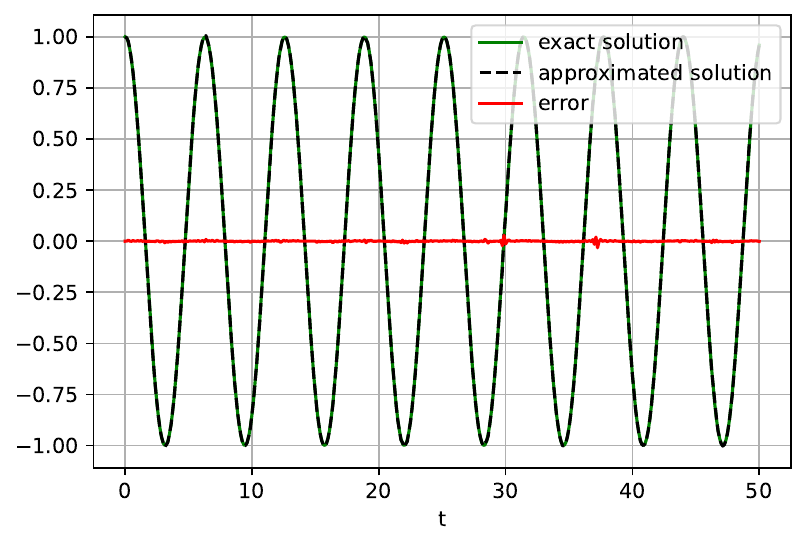}
        \subcaption{\eqref{eq:S1}, $h=\frac18$, path 3}
    \end{subfigure}
    \caption{Numerical solution of \eqref{eq:num-test} obtained via schemes \eqref{eq:detS1} and \eqref{eq:S1} for step size $h\in\bigl\{\frac12,\frac14,\frac18\bigr\}$.}
    \label{fig:stability-semiimpl}
\end{figure}

 \begin{figure}[h]
    \centering
    \begin{subfigure}{0.3\textwidth}
    \centering
        \includegraphics[width=.95\linewidth]{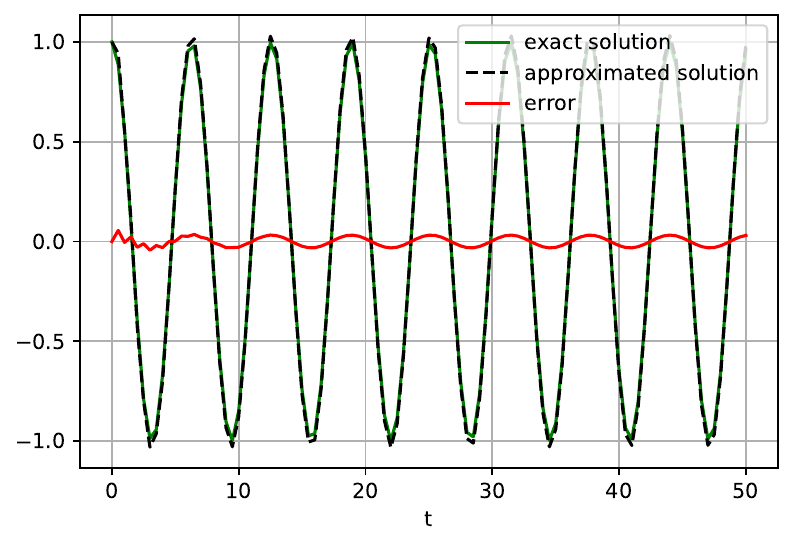}
        \subcaption{\eqref{eq:detS2}, $h=\frac12$}
    \end{subfigure}
    \begin{subfigure}{0.04\textwidth}
    \end{subfigure}
    \begin{subfigure}{0.3\textwidth}
    \centering
        \includegraphics[width=.95\linewidth]{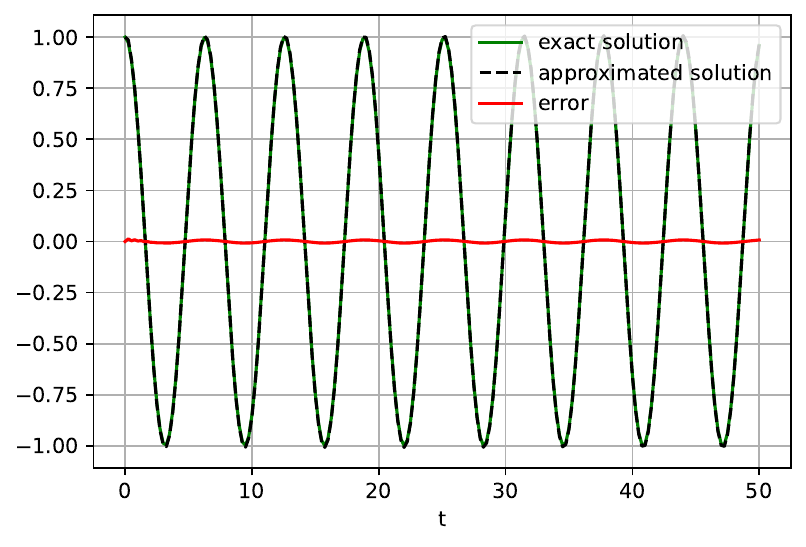}
        \subcaption{\eqref{eq:detS2}, $h=\frac14$}
    \end{subfigure}
    \begin{subfigure}{0.04\textwidth}
    \end{subfigure}
    \begin{subfigure}{0.3\textwidth}
    \centering
        \includegraphics[width=.95\linewidth]{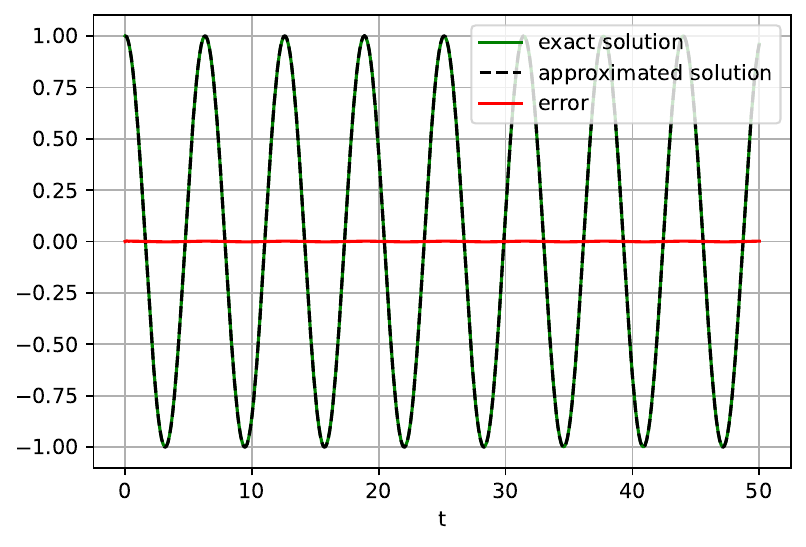}
        \subcaption{\eqref{eq:detS2}, $h=\frac18$}
    \end{subfigure}
    
    \begin{subfigure}{0.3\textwidth}
    \centering
        \includegraphics[width=.95\linewidth]{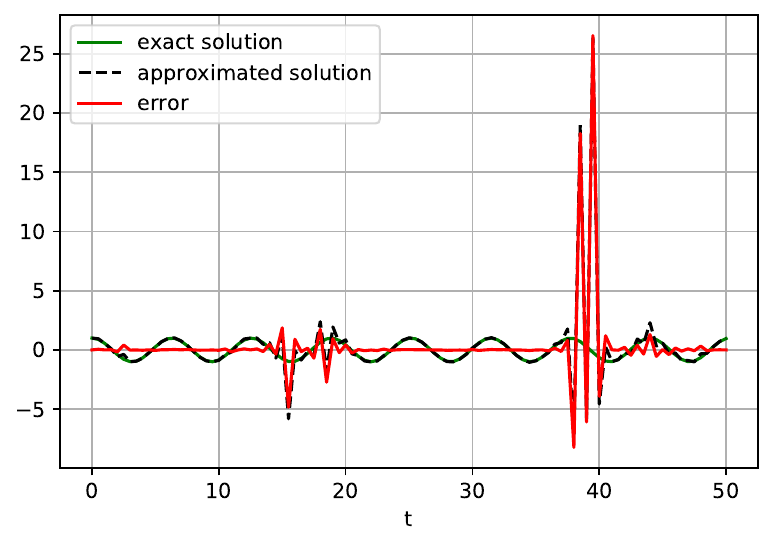}
        \subcaption{\eqref{eq:S2}, $h=\frac12$, path 1}
    \end{subfigure}
    \begin{subfigure}{0.04\textwidth}
    \end{subfigure}
    \begin{subfigure}{0.3\textwidth}
    \centering
        \includegraphics[width=.95\linewidth]{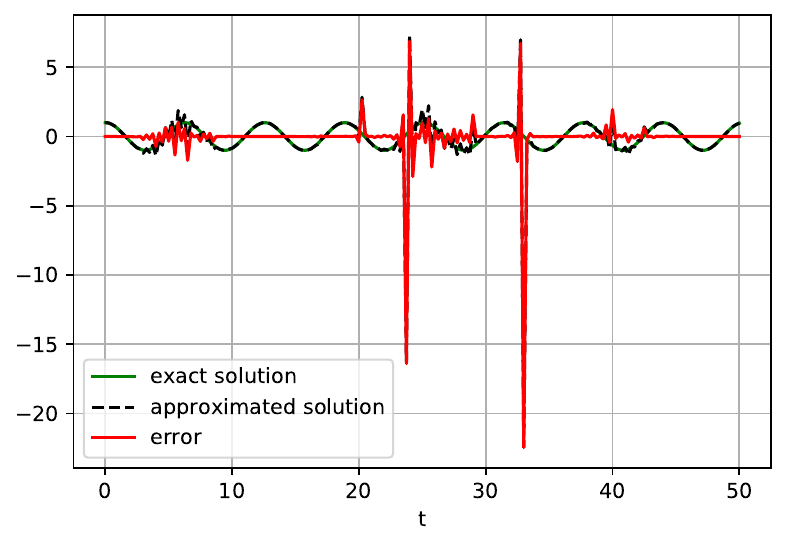}
        \subcaption{\eqref{eq:S2}, $h=\frac14$, path 1}
    \end{subfigure}
    \begin{subfigure}{0.04\textwidth}
    \end{subfigure}
    \begin{subfigure}{0.3\textwidth}
    \centering
        \includegraphics[width=.95\linewidth]{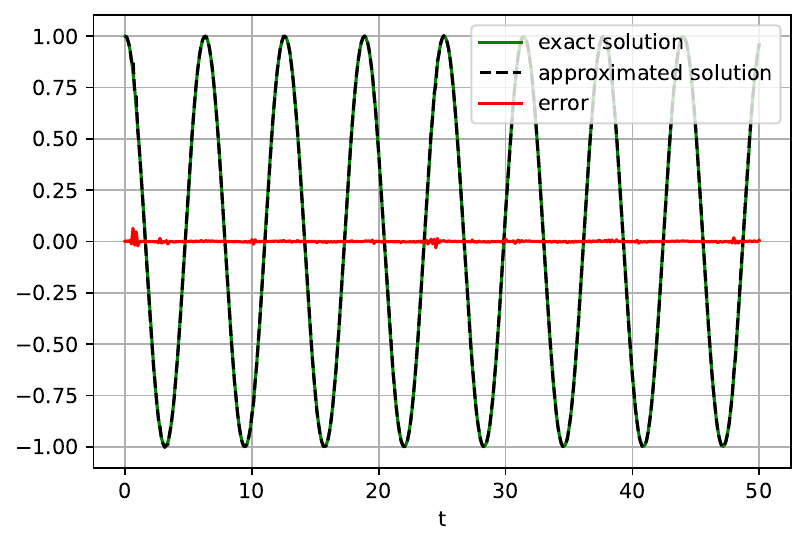}
        \subcaption{\eqref{eq:S2}, $h=\frac18$, path 1}
    \end{subfigure}
    
    \begin{subfigure}{0.3\textwidth}
    \centering
        \includegraphics[width=.95\linewidth]{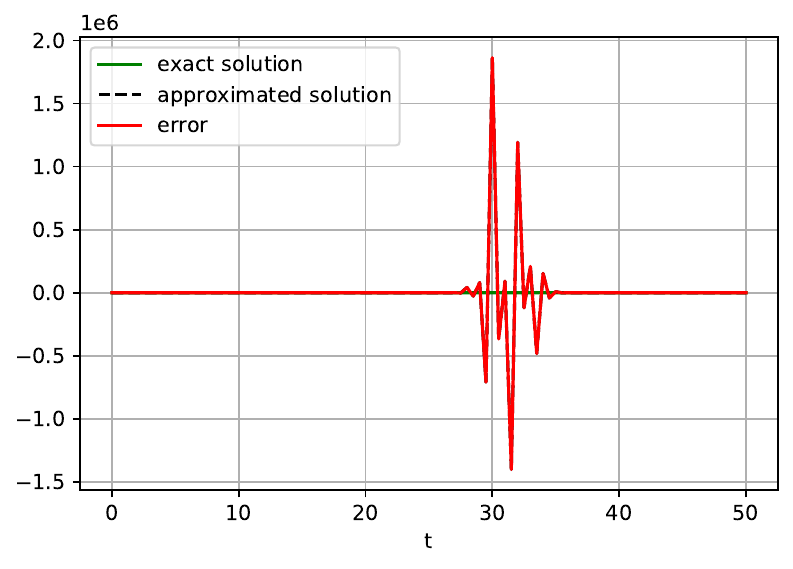}
        \subcaption{\eqref{eq:S2}, $h=\frac12$, path 2}
    \end{subfigure}
    \begin{subfigure}{0.04\textwidth}
    \end{subfigure}
    \begin{subfigure}{0.3\textwidth}
    \centering
        \includegraphics[width=.95\linewidth]{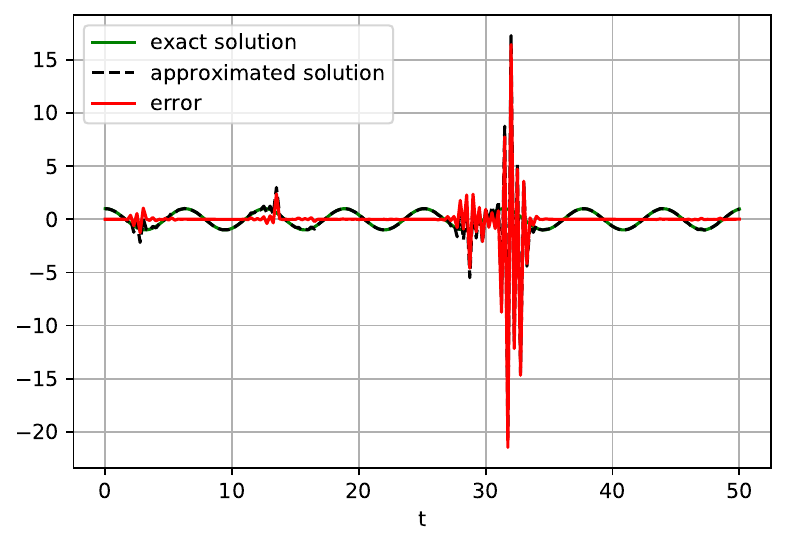}
        \subcaption{\eqref{eq:S2}, $h=\frac14$, path 2}
    \end{subfigure}
    \begin{subfigure}{0.04\textwidth}
    \end{subfigure}
    \begin{subfigure}{0.3\textwidth}
    \centering
        \includegraphics[width=.95\linewidth]{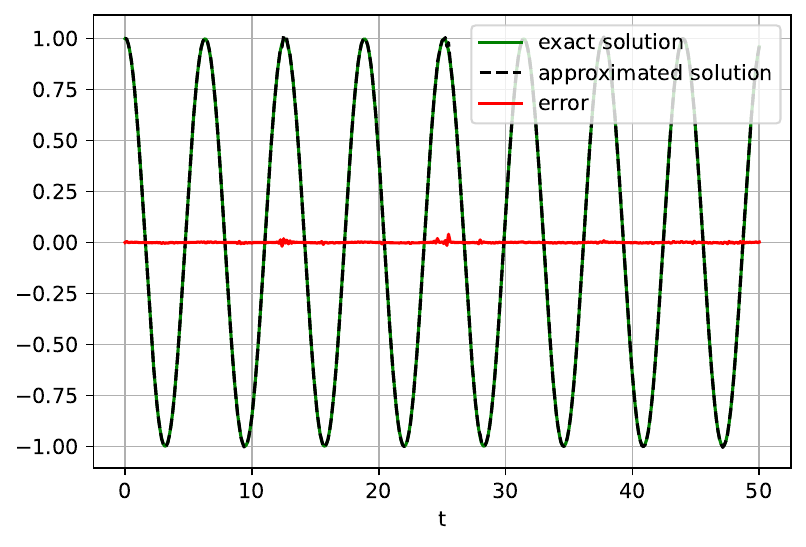}
        \subcaption{\eqref{eq:S2}, $h=\frac18$, path 2}
    \end{subfigure}
    
    \begin{subfigure}{0.3\textwidth}
    \centering
        \includegraphics[width=.95\linewidth]{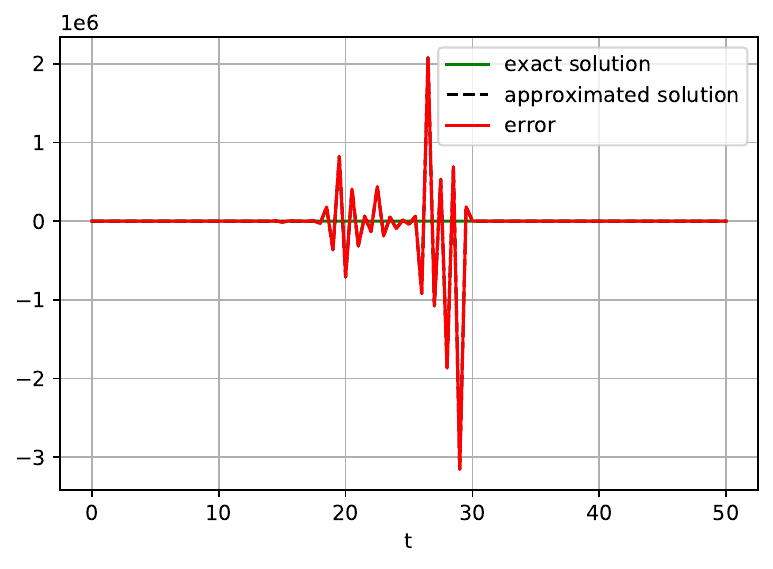}
        \subcaption{\eqref{eq:S2}, $h=\frac12$, path 3}
    \end{subfigure}
    \begin{subfigure}{0.04\textwidth}
    \end{subfigure}
    \begin{subfigure}{0.3\textwidth}
    \centering
        \includegraphics[width=.95\linewidth]{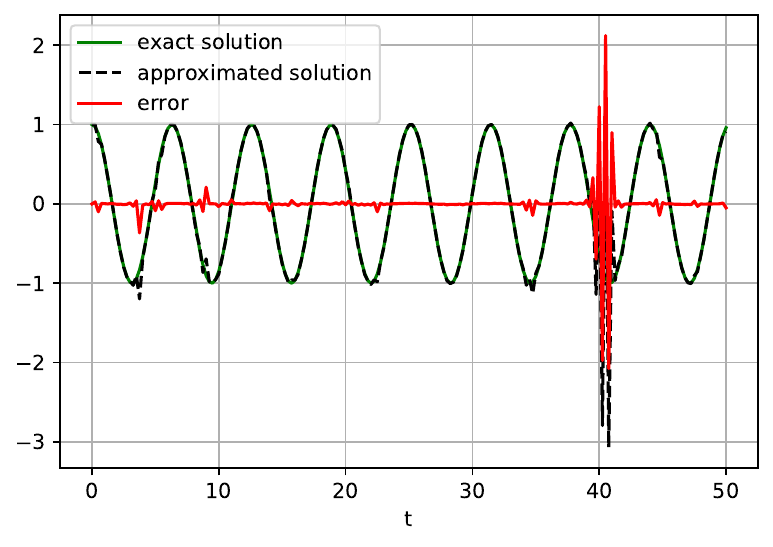}
        \subcaption{\eqref{eq:S2}, $h=\frac14$, path 3}
    \end{subfigure}
    \begin{subfigure}{0.04\textwidth}
    \end{subfigure}
    \begin{subfigure}{0.3\textwidth}
    \centering
        \includegraphics[width=.95\linewidth]{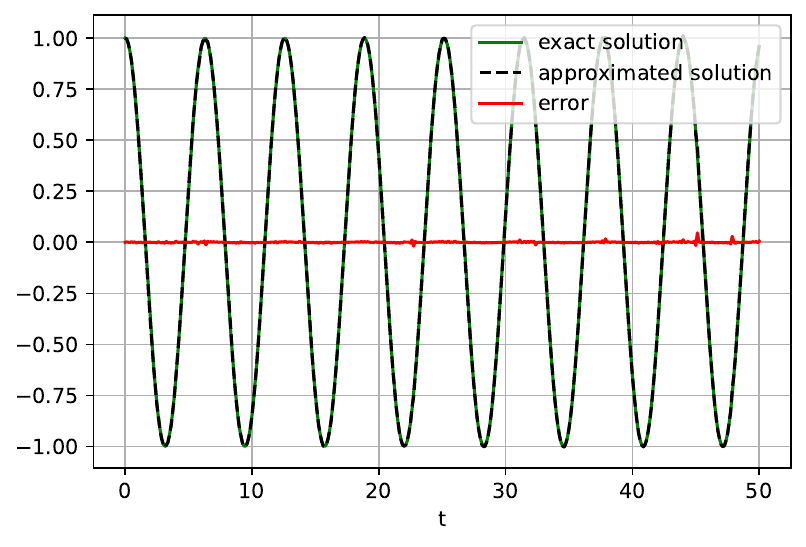}
        \subcaption{\eqref{eq:S2}, $h=\frac18$, path 3}
    \end{subfigure}
    \caption{Numerical solution of \eqref{eq:num-test} obtained via schemes \eqref{eq:detS2} and \eqref{eq:S2} for step size $h\in\bigl\{\frac12,\frac14,\frac18\bigr\}$.}
    \label{fig:stability-impl}
\end{figure}

\section{Conclusions and future work} \label{sec:conclusions}

In this paper, we have analyzed two randomized implicit schemes from the two-stage Runge-Kutta family. We have shown that they achieve the rate of convergence known for the randomized explicit RK2 scheme (Theorem \ref{theorem-error-1} and Theorem \ref{theorem-error-2}) which is by $1/2$ better than for analogous deterministic schemes. We have described probabilistic stability regions and proven that the schemes are A-stable asymptotically (Fact \ref{f:AS=SP}) but not in the mean-square sense (Fact \ref{f:MS:bounded}). Finally, via numerical experiments, we assessed the stability of the proposed schemes and compared their performance to similar deterministic and explicit schemes.

In our future work, we plan to investigate further extensions of the results obtained so far. In particular, we plan to analyze stability regions defined in the $L^p(\Omega)$ sense for all $p\in [1,+\infty)$ and to investigate their potential links to asymptotic stability. Moreover, the construction of randomized schemes of higher order, their errors, and stability properties are also of interest.
\ 

\bf Acknowledgments. \rm This research was funded in whole or in part by the National Science Centre, Poland, under project 2021/41/N/ST1/00135. 

This preprint has not undergone peer review or any post-submission improvements or corrections. The Version of Record of this article is published in BIT Numerical Mathematics, and is available online at https://doi.org/10.1007/s10543-024-01050-9.



\begin{thebibliography}{22}

\bibitem{taylor}
T. Bochacik, A note on the probabilistic stability  of randomized Taylor schemes, {\em Electron. Trans. Numer. Anal.} {\bf 58} (2023), 101--114.

\bibitem{randRK}
T. Bochacik, M. Goćwin, P. M. Morkisz, P. Przybyłowicz, Randomized Runge-Kutta method -- Stability and convergence under inexact information, {\em J. Complex.} {\bf 65} (2021), 101554.

\bibitem{randEuler}
T. Bochacik, P. Przybyłowicz, On the randomized Euler schemes for ODEs under inexact information, {\em Numer. Algorithms} {\bf 91} (2022), 1205--1229.

\bibitem{butcher}
J. C. Butcher, {\em Numerical Methods for Ordinary Differential Equations}, 2nd ed., Wiley, 2008. 

\bibitem{dahlquist}
G. Dahlquist, A special stability problem for linear multistep methods, {\em BIT Numer. Math.} {\bf 3} (1963), 27--43.

\bibitem{ambrosio}
R. D'Ambrosio, \emph{Numerical Approximation of Ordinary Differential Problems. From Deterministic to Stochastic Numerical Methods}, UNITEXT, Springer, 2023.

\bibitem{daun1}
T. Daun, On the randomized solution of initial value problems, {\em J. Complex.} \bf 27 \rm (2011), 300--311.

\bibitem{backward_euler} M. Eisenmann, M. Kov\'{a}cs, R. Kruse, S. Larsson, On a randomized backward Euler method for nonlinear evolution equations with time-irregular coefficients, {\em Found. Comp. Math.} {\bf 19} (2019), 1387--1430.

\bibitem{hairer}
E. Hairer and G. Wanner, \emph{Solving Ordinary Differential Equations II, Stiff and Differential-
Algebraic Problems, 2nd ed.}, Springer-Verlag, Berlin, 1996.

\bibitem{HeinMilla} 
S. Heinrich, B. Milla, The randomized complexity of initial value problems, {\em J. Complex.} \bf 24 \rm (2008), 77--88.

\bibitem{higham1}
D.J. Higham, A-stability and stochastic mean-square stability, {\em BIT  Numer. Math.} {\bf 40} (2000), 404--409.

\bibitem{higham2}
D. J. Higham, Mean-square and asymptotic stability of the stochastic theta method, {\em Siam J. Numer. Anal.} {\bf 38} (2000), 753--769.

\bibitem{JenNeuen}
A. Jentzen, A. Neuenkirch, A random Euler scheme for Carath\'eodory differential equations, {\em J. Comp. and Appl. Math.} \bf 224 \rm (2009), 346--359.

\bibitem{Kac1}
B. Kacewicz, Almost optimal solution of initial-value problems by randomized and quantum algorithms, {\em J. Complex.} \bf 22 \rm
(2006), 676--690.

\bibitem{KruseWu_1} 
R. Kruse, Y. Wu, Error analysis of randomized Runge–Kutta
methods for differential equations with time-irregular coefficients, {\em Comput. Methods Appl. Math.}, {\bf 17} (2017), 479--498.

\bibitem{mitsui}
T. Mitsui and Y. Saito, Stability Analysis of Numerical Schemes for Stochastic Differential Equations, {\em SIAM J. Numer. Anal.}, {\bf 33} (1996), 2254--2267.

\bibitem{stengle1}
G. Stengle, Numerical methods for systems with measurable coefficients, {\em Appl. Math. Lett.} {\bf 3} (1990), 25--29.

\bibitem{stengle2}
G. Stengle, Error analysis of a randomized numerical method, {\em Numer. Math.} {\bf 70} (1995) 119--128.

\end{thebibliography}
\end{document}